\documentclass[11pt]{amsart}
\usepackage{amsmath,amsthm,amsfonts,amssymb,mathrsfs,epsfig,bm}
\usepackage{geometry} 
\usepackage{amstext}
\usepackage{amsmath, amsthm} 
 
\numberwithin{equation}{section}

\newtheorem{thm}{Theorem}[section]
\newtheorem{cor}[thm]{Corollary}
\newtheorem{lem}[thm]{Lemma}
\newtheorem{prop}[thm]{Proposition}

\newcommand{\R}{\mathbb{R}}
\newcommand{\He}{\mathbb{H}}
\newcommand{\Z}{\mathbb{Z}}
\newcommand{\E}{\mathbb{E}}

\usepackage{color}
\usepackage[normalem]{ulem}
\usepackage{amssymb}

\begin{document}

\title[log-Sobolev  with  Quadratic Interactions.]{ The  log-Sobolev inequality with quadratic interactions.}

\author[Ioannis Papageorgiou]{Ioannis Papageorgiou }
\thanks{\textit{Address:} Departamento de Matematica, Universitad de Buenos Aires, Pabellon II, Ciudad Universitaria Buenos Aires, Argentina.  
\\ \text{\  \   \      } \textit{Email:}  ipapageorgiou@dm.uba.ar, papyannis@yahoo.com}

\begin{abstract}
 We assume one site measures without a boundary $e^{-\phi(x)}dx/Z$ that  satisfy a log-Sobolev inequality. We prove that if these  measures are perturbed with quadratic interactions, then the associated  infinite dimensional Gibbs measure on the lattice always satisfies a log-Sobolev inequality. Furthermore, we present examples of measures that satisfy the inequality with a phase  that goes beyond convexity at infinity.
\end{abstract}
\maketitle

\section{Introduction}  
We focus on the logarithmic Sobolev inequality for unbounded  spin systems on the d-dimensional lattice $\Z^d$ ($d\geq 1$)    with  quadratic interactions. The aim of this paper is to prove that when the single site measure without interactions  (consisting only of the phase) 
$$\mu(dx)=\frac{e^{-\phi(x)}dx}{\int e^{-\phi(x)} dx}$$  
satisfies the log-Sobolev inequality, then the Gibbs measure of the associated local specification $\left\{\E^{\Lambda,\omega}\right\}_{\Lambda\subset \Z^d,\omega \in M^{\partial \Lambda}}$, with Hamiltonian
$$H^{\Lambda,\omega}(x_\Lambda) 
:= \sum_{i \in \Lambda} \phi(x_i) 
+ \sum_{i,j \in \Lambda,\, j \sim i} J_{ij} V(x_i, x_j)
+ \sum_{i\in \Lambda,j \in \partial \Lambda,\, j \sim i} J_{ij} V(x_i, \omega_j),
$$    also satisfies a log-Sobolev inequality, when the interactions $V$ are quadratic.

Since the main condition about the phase measure does not involve the local specification  $\{\E^{\Lambda}\}$ nor  the one site measure $\mathbb{E}^{\{i\},\omega}$, we present a criterion for the infinite dimensional Gibbs measure inequality without assuming or proving  the usual Dobrushin and Shlosman's mixing conditions for the local specification as in \cite{S-Z} and more recently in \cite{M2}.  As a matter of fact, in order to control the boundary conditions involved in the interactions, we will make  use of the U-bound inequalities introduced in \cite{H-Z} to prove coercive inequalities in a standard non statistical mechanics framework. As a result we prove the inequality for a variety of phases extending beyond the usual Euclidean case,  as well as involving measures with phase like  $$\phi(x)=d^p(x)+cos(d(x))d^{p-1}(x),   \   \  p\geq2$$  that go beyond the typical convexity at infinity.

For the investigation of criteria for the logarithmic Sobolev inequality for  the infinite dimensional Gibbs measure of the local specification $\left\{\E^{\Lambda,\omega}\right\}_{\Lambda\subset \Z^d,\omega \in M^{\partial \Lambda}}$ two  main approaches  have been developed. 

The first approach is based in proving first  that the measures $\left\{\E^{\Lambda,\omega}\right\}_{\Lambda\subset \Z^d,\omega \in M^{\partial \Lambda}}$ satisfy a log-Sobolev inequality with a constant uniformly on the set $\Lambda$ and the boundary conditions $\omega$. Then the inequality for the Gibbs measure follows directly from the uniform inequality for the local specification. Criteria for the local specification to satisfy a log-Sobolev inequality uniformly on the set $\Lambda$ and the boundary conditions have been investigated by   \cite{Z2},   \cite{B},  \cite{B-E},  \cite{Y},  \cite{A-B-C}, 
 \cite{B-H} and \cite{H}.
Similar results for the weaker spectral gap inequality have been obtained by  \cite{G-R}. 

The second approach focuses in obtaining the inequality for the Gibbs measure directly, without showing first the stronger uniform inequality for the local specification. Such criteria on the local specification in the case of quadratic interactions  for the infinite-dimensional Gibbs measure on the
lattice have been investigated by    \cite{Z1}, \cite{Z2} and    \cite{G-Z}.

The problem of passing from the single site to infinite dimensional measure,
in the case of quadratic interactions,
is addressed by   \cite{M1},   \cite{I-P} and 
  \cite{O-R}. What it has been shown is that when the one site measure  $\mathbb{E}^{\{i\},\omega}$ satisfies a log-Sobolev inequality uniformly on the boundary conditions, then in the presence of quadratic  interactions the infinite Gibbs measure also satisfies a log-Sobolev inequality.

For the single-site measure  
$\mathbb{E}^{\{i\},\omega}$, necessary and sufficient  conditions   for
the log-Sobolev inequality to be satisfied uniformly over the boundary
conditions $\omega$, are also presented in    \cite{B-G}, 
  \cite{B-Z} and   \cite{R-Z}.

The scope of the current paper is to  prove the log-Sobolev inequality for the Gibbs measure without setting conditions  neither on the local specification $\{\E^{\Lambda}\}$ nor on the one site measure $\mathbb{E}^{\{i\},\omega}$. What we actually show is that in the presence of quadratic interactions, the Gibbs measure always satisfies a log-Sobolev inequality whenever the boundary free one site measure $\mu(dx)=e^{-\phi(x)}dx /(\int e^{-\phi(x)} dx)$  satisfies a log-Sobolev inequality.
 In that way we improve the previous results since the log-Sobolev inequality is determined alone by the phase  $\phi$ of the simple without interactions measure $\mu$ on $M$, for which a plethora of criteria and examples of good measure that satisfy the inequality exist.
\section{General framework and main result.}

We consider the $d$-dimensional integer lattice $\Z^d$  with 
the standard neighborhood
notion, where  two lattice points  $i, j \in \Z^d$ are considered neighbours if their lattice distance is one, i.e. they are connected with an edge, in which case we write   $i \sim j$. We will also denote $\{\sim i\}$ for the set of all neighbours of a node $i$ and $\partial \Lambda$ the boundary of a set $\Lambda\subset \Z^d$.  Our configuration space is $\Omega=M^{\Z^d}$, where $M$ is the spin space. 

We consider unbounded $n$-dimensional  spin spaces $M$ with the following structure. 
We shall assume that $M$ is a nilpotent  Lie group on
$\R^n$ with a  H\"ormander system $X^1, \ldots, X^N$, $N \le n$,
of smooth vector fields $X^k = \sum_{j=1}^n a_{kj} \frac{\partial}{\partial x_j}$,
$k=1,\ldots, N$, i.e. $a_{kj}$ are smooth functions of $x \in \R^n$.
The (sub)gradient $\nabla$ with respect   
to this structure is the vector operator $\nabla f = (X^1 f, \ldots, X^N f)$.
We consider $$\|\nabla f\|^2 := (X^1 f)^2 + \cdots + (X^N f)^2$$
When these operators refer to a  spin space $M^i$ at a node $i \in \Z^d$ 
this will be  indicated by an index $\nabla_if=(X^1 _{i}f, \ldots, X^N _{i}f)$. 
For    a subset  $\Lambda$  of  $\Z^d$ we define 
$\nabla_\Lambda := (\nabla_i, i \in \Lambda)$ and 
$$\|\nabla_\Lambda f\|^2 := \sum_{i \in \Lambda} \|\nabla_i f\|^2$$

The spin space $M$ is equipped with a metric
 $d(x,y)$ for $x,y \in M$.
 For example, in the case of  $M$ being a Euclidean space then $d$ is the Euclidean
metric or if $M$ is the Heisenberg group, 
then $d$ is the Carnot-Carath\'eodory metric. We will consider examples and applications of the main theorem  for both. 
In  all cases, for $x \in M$  we will conventionally write  $d(x)$, 
for the distance of $x$ from  $0$ $$d(x):=d(x, 0)$$ where $0$ is a specific point of $M$, for
example the origin if $M$ is $\R^n$ or the identity element of the group
when $M$ is a Lie group.
 Furthermore, we assume that there exists a $k_0>0$ such that $\|\nabla d\|\leq k_0$. For instance, in the Euclidean and the Carnot-Caratheodory metrics $k_0=1$.

A spin at a site $i\in \Z^d$    of a configuration $\omega\ \in \Omega$ will be indicated by an index, i.e. we will write $\omega_i$. This takes values in $M^i$ which is an identical copy of the spin space $M$.
For a subset $\Lambda \subset \Z^d$ we will identify $M^\Lambda$
with the Cartesian product of the $M^i$ for every $i\in\Lambda$.

The spin space $M$ is equipped with a natural measure. For example,
when $M$ is a group then we assume that the measure is one which is invariant under
the group operation, for which  we write $dx$. Again, for any $i \in \Z^d$, we use a subscript to indicate   the  natural measure $dx_i$ on
$M^i$. In the case of a Euclidean space or the Heisenberg group for instance, this is the Lebesgue measure. For the product measure derived from  the $dx_i$, $i \in \Lambda$ we will write   $dx_\Lambda:=\otimes _{i\in \Lambda}dx_{i}$.  The measures of the local specification  $\left\{\E^{\Lambda,\omega}\right\}$ for $\Lambda\subset \Z^d$ and $\omega \in M^{\partial \Lambda}$, are defined as\[
\E^{\Lambda, \omega} (dx_\Lambda) = \frac{1}{Z^{\Lambda,\omega}}\,
e^{-H^{\Lambda,\omega}(x_\Lambda)}\, dx_\Lambda,
\]
where $Z^{\Lambda,\omega}$ is a normalization constant. The Hamiltonian function $H^{\Lambda, \omega}$ has
the form
\[
H^{\Lambda,\omega}(x_\Lambda) 
:= \sum_{i \in \Lambda} \phi(x_i) 
+ \sum_{i,j \in \Lambda,\, j \sim i} J_{ij} V(x_i, x_j)
+ \sum_{i\in \Lambda,j \in \partial \Lambda,\, j \sim i} J_{ij} V(x_i, \omega_j),
\]
We call $\phi$ the phase and $V$ the interaction. In this work we consider exclusively  quadratic interactions $V$, i.e. 
\begin{align}\label{quadratic}\nonumber \left\vert V(x_i,\omega_j)\right\vert
&\leq kd^2(x_i)+k d^2(\omega_j) \\  & \text{and}\\  \nonumber\|\nabla_iV(x_i,\omega_j)\|^2 &\leq kd^2(x_i)+k d^2(\omega_j) \end{align} for some $k\geq 1$.  We will assume that there exists a $J$ such that $\left\vert J_{ij}\right\vert\leq J$ and that $J_{ij}V(x_i,x_j)\geq 0$.

For a function  $f$  from $M^{\Z^d}$ into $\R$,
we will conventionally write   $\E^{\Lambda, \omega} f$ for the expectation of $f$ with respect  
to $\E^{\Lambda, \omega}$.
For economy we will frequently omit the boundary conditions and  we will write  $\E^\Lambda f$ instead of $\E^{\Lambda, \omega}$.

The measures of the local specification obey the Markov property 
\[
\E^\Lambda \E^K f= \E^\Lambda f, \quad K \subset \Lambda.
\]
We say that the probability measure $\nu$ on $\Omega = M^{\Z^d}$
is an infinite volume Gibbs measure for the local
specifications $\{\E^{\Lambda,\omega}\}$ if
it satisfies the  Dobrushin-Lanford-Ruelle equation:
\[
\nu \mathbb{E}^{\Lambda,\bullet}=\nu, \quad \Lambda \Subset \Z^d,
\]
 We refer to   \cite{Pr}, \cite{B-HK} and \cite{D} for details.
Throughout the paper we shall assume that we are in the case
where $\nu$ exists   (uniqueness will be
deduced from our   results, see Proposition \ref{7prop2}).  Furthermore, we will consider functions $f:M^{\Z^d}\rightarrow\R$ such that $fd\in L_2(\nu)$.

The  main interest  of the paper is the logarithmic Sobolev inequality. We say that a  probability measure $\mu$ in $M$ satisfies the logarithmic Sobolev inequality, if  there exists a constant $c>0$ such that
\begin{align}
\label{LS}
\mu\bigg( f^2 \log \frac{f^2}{\mu (f^2)}
\bigg) \le c\mu\|\nabla f\|^2
\end{align} 
We notice two important properties for the log-Sobolev inequality. The first is that it implies the  
spectral gap inequalities, that is,     there exists a constant $c'<c$ such that  $$\mu\ |f-\mu f|^2 \le c'\, \mu\ \|\nabla f\|^2$$
The second is that    both the log-Sobolev inequality and the spectral gap inequality are retained under product measures. Proofs of these two assertions can be found in  
Gross \cite{G}, Guionnet and Zegarlinski \cite{G-Z} and 
Bobkov and Zegarlinski \cite{B-Z}.

Under the spin system framework  the   log-Sobolev inequality for the local specification $\{\E^{\Lambda,\omega}\}$ takes the form 
\begin{align}
\label{LSE}
\E^{\Lambda,\omega}\bigg( f^2 \log \frac{f^2}{\mathbb{E}^{\Lambda,\omega} f^2}
\bigg) \le c\, \mathbb{E}^{\Lambda,\omega} \|\nabla_\Lambda f\|^2  
\end{align}
where the constant $c$ is now required uniformly
on the subset $\Lambda$ and the boundary conditions  $\omega \in \partial \Lambda$. In the special case where $\Lambda=\{i\}$ then the constant is considered uniformly on the boundary conditions $\omega \in \{ \sim i\}$. 
The analogue log-Sobolev inequality for the infinite volume Gibbs measure $\nu$ is then defined as 
\begin{align}
\label{LSG}
\nu\bigg( f^2 \log \frac{f^2}{\nu f^2}
\bigg) \le c\, \nu \|\nabla_{\Z^d} f\|^2  
\end{align}
The aim of this paper is to  show that the infinite volume Gibbs measure $\nu$ satisfies the log-Sobolev inequality (\ref{LSG}) for an appropriate constant. As explained in the introduction, in the case of quadratic interactions, previous works concentrated in proving first the stronger (\ref{LSE}) for all  $\Lambda\subset \Z^d$,   or assumed the log-Sobolev inequality (\ref{LSE}) for the one site $\Lambda=i$ and then  derived from these (\ref{LSG}).

Our aim is to show that    if we assume the weaker inequality (\ref{LS}) for the phase measure $\mu(dx)=\frac{e^{-\phi(x)}dx}{\int e^{-\phi(x)} dx}$, then in the presence of quadratic interaction this is sufficient to obtain directly the log-Sobolev inequality for the Gibbs measure (\ref{LSG}),    without the need to assume or prove any of the stronger inequalities (\ref{LSE}) that require uniformity on the boundary conditions and/or  the dimension of the measure.

The first result of the paper follows:

\begin{thm}\label{theorem}
 Assume   that the measure $\mu(dx)=\frac{e^{-\phi(x)}dx}{\int e^{-\phi(x)} dx}$  in $ M$ satisfies the log-Sobolev inequality and that the  local specification $\{\mathbb{E}^{\Lambda,\omega}\}$ has quadratic interactions $V$ as in (\ref{quadratic}). Then   for $J$ sufficiently small    the infinite dimensional Gibbs measure in $M^{\Z^d}$ satisfies a log-Sobolev inequality.\end{thm} 
Since the main hypothesis of the theorem refers just to the measure $\mu(dx)=\frac{e^{-\phi(x)}dx}{\int e^{-\phi(x)} dx}$ satisfying a logarithmic Sobolev inequality, we can take all the probability measures from $\R^n$ that satisfy a log-Sobolev inequality and get measures on the statistical mechanics framework of spin systems on  the lattice $\Z^d$ just by adding quadratic interactions as described in (\ref{quadratic}).

From the plethora of theorems and criteria that have been developed for the  Euclidean  $\R^n$ for $n\geq1$, among others in \cite{B-L}, \cite{B}, \cite{B-E}, \cite{B-G} and \cite{B-Z}  one can  generalise these to the spin system framework just by applying them to the phase $\phi$ and then add quadratic interactions $V$. As a typical example, one can then for instance obtain for then 
 Euclidean  space $n\geq1$ with $d$  the Euclidian metric,  $X_i=\frac{\partial}{\partial x_j}$ and $dx_i$ the Lebesgue measure, the following example of measures: Consider the phase $\phi(c)=\| x \|_p^p$ for any $p\geq 2$ and interactions $V(x,y)=\| x-y \|_2^2$. Then the associated Gibbs measure satisfies a logarithmic Sobolev inequality.

Furthermore, as  will be described in Theorem \ref{theorem2} that follows, with additional assumptions on the distance and the gradient we can obtain results comparable to the once obtained in \cite{H-Z} for general metric spaces.   We consider general $n$-dimensional non compact metric spaces. For the distance $d$  and the (sub)gradient   $\nabla$, in addition to the hypothesis of Theorem \ref{theorem}   we assume that $$(D1):\frac{1}{\sigma}<\vert \nabla d \vert\leq 1    $$ for some $\sigma \in [1,\infty)$, and $$(D2):\Delta d\leq K  $$ outside the unit ball $\{d(x)<1\}$ for some $K\in(0,+\infty)$. We also assume that the gradient $\nabla$ satisfies the integration by parts formula. In the case of $\nabla f = (X^1 f, \ldots, X^N f)$ with vector fields  $X^k = \sum_{j=1}^n a_{kj} \frac{\partial}{\partial x_j}$ it suffices to request that $a_{kj}$ is a function of $x \in \R^d$
not depending on the $j$-th coordinate $x_j$.

 If  $dx$ is the $n-$dimensional Lesbegue measure we assume that it satisfies the  Classical-Sobolev inequality (C-S)   
$$ \left(\int \vert f\vert ^{2+\epsilon}dx\right)^\frac{2}{2+\epsilon}\leq\alpha \int\|\nabla f  \|^2dx +\beta\int \vert f \vert ^2dx  \   \   \   \   \   \   \   \   \   \       \text{(C-S)}$$ for positive constants $\alpha, \beta$, as well as    the Poincar\'e inequality on the ball $B_R$, that is there exists a constant $c_R\in(0,\infty)$ such that \begin{align*}\frac{1}{\vert B_R\vert}\int_{B_R}\left \vert f-\frac{1}{\vert B_R\vert}\int_{B_R}f\right \vert^2dx\leq c_R\frac{1}{B_R} \int_{B_R} \| \nabla f  \|^2dx \   \   \   \   \   \   \   \       \text{(L-P)} \end{align*}
  The Classical Sobolev inequality (C-S)  is for instance satisfied in the case of the $\mathbb{R}^n,n\geq 1$ with $d$ being the Eucledian distance, as well as for the case of the Heisenberg group, with $d$ being the Carnot-Carath\'eodory distance.   The Poincar\'e inequality on the ball for the Lebesgue measure (L-P) is a standard result for $n\geq 3$ (see for instance \cite{H}, \cite{H-Z}, \cite{D} and \cite{V-SC-C}), while for $n=1,2$ one can look on  \cite{Pa2}.  Under this framework, if
 we  combine our main result  Theorem \ref{theorem}, together with
  Corollary 3.1 and Theorem 4.1 from \cite{H-Z} we  obtain the  following theorem
\begin{thm}\label{theorem2}Assume  distance $d$  and the (sub)gradient   $\nabla$  are such that (D1)-(D2) as well as  (C-S) and (L-P) are satisfied. Let a probability measure $\mu(dx)=\frac{e^{-\phi(x)}dx}{\int e^{-\phi(x)}dx}$, where $dx$ the Lebesgue measure, such that  $$\phi(x)= W(x)+B(x)$$ defined with a differential potential $W$ satisfying 
\begin{align*}\|\nabla W \|^q \leq \delta d^p+\gamma
\end{align*}with $p\geq 2$ and $q$ the conjugate of $p$,  and  suppose that $B$ is a measurable function such that $osc(B)=\max B-\min B<\infty$. Assume that the  local specification $\{\mathbb{E}^{\Lambda,\omega}\}$ has quadratic interactions $V$ as in (\ref{quadratic}). Then    for $J$ sufficiently small    the infinite dimensional Gibbs measure satisfies a log-Sobolev inequality.
\end{thm}
 
An interesting application of the last theorem  is the  special case of the   Heisenberg group, $\mathbb{H}$. This can be described as $\mathbb{R}^3$ with the following group operation:
$$x\cdot\tilde{x} = (x_1, x_2, x_3)\cdot(\tilde{x}_1, \tilde{x}_2, \tilde{x}_3) = (x_1 + \tilde{x}_1, x_2 + \tilde{x}_2, x_3 + \tilde{x}_3 + \frac{1}{2}(x_1\tilde{x}_2 - x_2\tilde{x}_1))$$
$\mathbb{H}$ is a Lie group, and its Lie algebra $\mathfrak{h}$ can be identified with the space of left invariant vector fields on $\mathbb{H}$ in the standard way.  The vector fields\begin{eqnarray}
X_1 &=& \partial_{x_1} - \frac{1}{2}x_2\partial_{x_3} \nonumber \\
X_2 &=& \partial_{x_2} + \frac{1}{2}x_1\partial_{x_3} \nonumber \\
X_3 &=& \partial_{x_3} = [X_1, X_2] \nonumber
\end{eqnarray}
form a Jacobian  basis,   where $\partial_{x_i}$ denoted derivation with respect to $x_i$.
From this it is clear that $X_1, X_2$ satisfy the H\"ormander condition 
(i.e., $X_1, X_2$ and their commutator $[X_1, X_2]$ span the tangent space 
at every point of $\He_1$).  The sub-gradient is given by
$$\nabla := (X_1,X_2)$$
   For more details one  can look at \cite{B-L-U}.   In \cite{I-P} a first example of a measure on the Heisenberg group with a Gibbs measure that satisfies   a logarithmic Sobolev inequality was presented. Here, with the use of Theorem \ref{theorem2} we can obtain examples with a phase $\phi$ that is nowhere convex and include  more natural quadratic interactions.  Such an example, that satisfies the conditions of Theorem \ref{theorem2} with a phase that goes beyond convexity at infinity is the following:
$$\phi(x)=d^p(x)+cos(d(x))d^{p-1}(x)$$ and $$V(x,y)=d^{2}(x \cdot\ y^{-1})$$ where $\cdot$ the group operation and $y^{-1}$ the inverse in respect to this operation.

The proof of Theorem \ref{theorem}  is divided   into two parts presented on the next two propositions \ref{proposition} and \ref{propUbound}. In the first  one, we prove a weaker assertion, that the claim of Theorem \ref{theorem} is true  under the conditions of Theorem \ref{theorem} together with the U-bound inequality (\ref{Ubound}).    
  \begin{prop}\label{proposition} Assume   that   the measure $\mu(dx)=\frac{e^{-\phi(x)}dx}{\int e^{-\phi(x)} dx}$ satisfies the log-Sobolev inequality and  that the  local specification $\{\mathbb{E}^{\Lambda,\omega}\}$ has quadratic interactions $V$ as in (\ref{quadratic}). Furthermore, assume that       there exists a $C\geq 1$ such that  the following U-bound inequality is satisfied
\begin{align}\label{Ubound}  \nu ( d^2(x_k)f^2)\leq\   C  \nu( f^{2})+C \sum_{n=0}^{\infty} J^{n}\sum_{j:dist(j,k)=n}\   \nu(\|\nabla_{j} f \|^2   )
\end{align} Then for $J$ sufficiently small the infinite dimensional Gibbs measure satisfies the log-Sobolev inequality. \end{prop}
The proof of this proposition  will be presented in section \ref{finalproof}.
 Then Theorem \ref{theorem} follows from Proposition \ref{proposition} and the next proposition  which states that the conditions of  Theorem \ref{theorem} imply the U-bound inequality (\ref{Ubound})  of Proposition \ref{proposition}. 
\begin{prop}\label{propUbound}
 Assume   that the measure $\mu(dx)=\frac{e^{-\phi(x)}dx}{\int e^{-\phi(x)} dx}$ satisfies the log-Sobolev inequality and  that the  local specification $\{\mathbb{E}^{\Lambda,\omega}\}$ has quadratic interactions $V$ as in (\ref{quadratic}).  Then for $J$ sufficiently small   the Gibbs measure satisfies the  U-bound inequality (\ref{Ubound}). \end{prop}
A few words about the structure of the paper. Since the proof of the main result presented in Theorem \ref{theorem} trivially follows from Proposition \ref{proposition} and Proposition  \ref{propUbound}, we concentrate on showing the validity of these two. 

For simplicity we will present the proof for the 2-dimensional lattice $\Z^2$. At first, the proof of 
  Proposition  \ref{propUbound} will be presented in section \ref{sectionUbound} where the U-bound inequality (\ref{Ubound}) is shown to hold under the conditions of the main theorem. 
The proof of Proposition \ref{proposition} will occupy the rest of the paper. In particular, in section \ref{section4} a Sweeping Out inequality will be shown as well as a spectral gap type inequality for the one site measure. In section \ref{secLS} a second Sweeping Out inequality is proven. In section \ref{proof sec6} logarithmic Sobolev type inequalities for the one site measure as well as for the infinite product measure are proven. Then in the section \ref{spectralgap} we present a spectral gap type inequality for the product measure directly from the log-Sobolev inequality shown in the previous section. Using this we show convergence to the Gibbs measure as well as it's uniqueness. Then at the final part of the section, in subsection \ref{finalproof},  we put all the previous bits together to prove Proposition \ref{proposition}.    \section{\label{sectionUbound}proof of the U-bound inequality}
 U-bound inequalities where introduced in \cite{H-Z} in order to prove $q$ logarithmic Sobolev inequalities. In this work we  use U-bound inequalities in order to control the quadratic interactions.  In this section we prove Proposition \ref{propUbound}, that states that if the measure $\mu(dx)=\frac{e^{-\phi(x)}}{\int e^{-\phi(x)} dx}$ satisfies the log-Sobolev inequality and the local specification has quadratic interactions then the U-bound inequality (\ref{Ubound})  is satisfied.

\begin{lem}\label{lemU1}If   $\mu $ satisfies the log-Sobolev inequality and    the  local specification has quadratic interactions $V$ as in (\ref{quadratic}), then for any $i\in  \mathbb{Z}^2$  \begin{align}\label{induction} \nu ( d^2(x_i)f^2)\leq K_{0}\nu( f^{2})+K_{0}\nu(\|\nabla_{i} f\|^2)+K_{1}&J^{2} \sum_{j\sim i}\nu(f^{2}d^2(\omega_j))
\end{align}for   positive constants $K_0$ and $K_{1}$.
\end{lem}

\begin{proof}
If we use the following entropic    inequality
(see \cite{D-S})
\begin{equation}\label{3eq2}\forall t>0, \ \pi(uv)\leq\frac{1}{t}log\left(\pi(e^{tu})\right)+\frac{1}{t}\pi(vlogv)\end{equation}
 for any probability measure $\pi$  and $v\geq 0$, $\pi v=1$, we get
 \begin{align}\label{3eq3}\mu( d^2(x_i)f^2)\leq\frac{1}{t}\mu( f^{2} )\log\left(\mu(e^{t d^2(x_i)})\right)      +\frac{1}{t}\mu(f^2log\frac{f^2}{\mu f^{2} })
\end{align}
For the first term on the right hand side of (\ref{3eq3}) we can use Theorem 4.5 from \cite{H-Z} (see also \cite{A-S}) which states that when a measure $\mu$ satisfies the log-Sobolev inequality then for any function  $g$ such that 
$$\|\nabla g\|^2\leq ag+b$$for $a,b\in(0.\infty)$ we have $$\mu e^{tg}<\infty$$ for all $t$ sufficiently small. Since 
$$\|\nabla (d^2)\|^2\leq 4d^2 \|\nabla d\|^{2}\leq 4k^{2}_0 d^2$$
from our hypothesis on $d:\|\nabla d\|\leq k_0$, we obtain that for $t$ sufficiently small 
$$\log\left(\mu(e^{t d^2(x_i)})\right)      \leq K$$ for some $K$.
From  this and the fact that  $\mu$ satisfies the log-Sobolev inequality (\ref{LS}) with some constant $c$, (\ref{3eq3}) becomes \begin{align*}\mu( d^2(x_i)f^2)\leq\frac{K}{t}\mu( f^{2} )+\frac{c}{t}\mu(\|\nabla_{i} f\|^2  )
\end{align*}
 If we substitute $f$ by $fe^{\frac{-V^{i}}{2}}$, where  we denoted $V^i=\sum_{ j\sim i}J_{i,j}V(x_{i},\omega_{j})$, we get \begin{align}\label{3eq4}\int e^{-H^{i,\omega}} d^2(x_i)f^2dx_{i}\leq\frac{K}{t}\int e^{-H^{i,\omega}}  f^{2} dx_{i}+\frac{c}{t}\mu(\|\nabla_{i} (fe^{\frac{-V^{i}}{2}})\|^2  )
\end{align}
For the second term of the right hand side of (\ref{3eq4}) we have   \begin{align*}\mu(\|\nabla_{i} (fe^{\frac{-V^{i}}{2}})\|^2  )\leq 2\int e^{-H^{i,\omega}} \|\nabla_{i} f\|^2dx_{i}+\frac{1}{2}\int e^{-H^{i,\omega}}f^{2}\|\nabla_{i} V^{i}\|^2dx_{i}    
\end{align*}
If we substitute this on (\ref{3eq4}) and divide both parts with  $Z^{i,\omega}$   we will get \begin{align*}\mathbb{E}^{i,\omega} ( d^2(x_i)f^2)\leq\frac{K}{t}\mathbb{E}^{i,\omega} ( f^{2})+\frac{2c}{t}\mathbb{E}^{i,\omega}(\|\nabla_{i} f\|^2)+\frac{c}{2t}\mathbb{E}^{i,\omega}(f^{2}\|\nabla_{i} V^{i}\|^2)
\end{align*}
 If we take  the expectation with respect to the Gibbs measure we obtain  \begin{align*}\nu( d^2(x_i)f^2)\leq\frac{K}{t}\nu( f^{2})+\frac{2c}{t}\nu(\|\nabla_{i} f\|^2)+\frac{c}{2t}\nu(f^{2}\|\nabla_{i} V^{i}\|^2)
\end{align*}
 From our main assumption (\ref{quadratic}) about the interactions, $\|\nabla_iV(x_i,\omega_j)\|^2\leq kd^2(x_i)+k d^2(\omega_i)$, we have that  $$\|\nabla_{i} V^{i}\|^2\leq 16k^{2}J^{2}d^2(x_i)+4k^{2}J^{2}\sum_{j\sim i}d^2(\omega_j)$$
which leads to 
 \begin{align*}\nu ( d^2(x_i)f^2)\leq\frac{K}{t}\nu( f^{2})+\frac{2c}{t}\nu(\|\nabla_{i} f\|^2)+&\frac{8ck^{2}J^{2}}{t}\nu(f^{2}d^2(x_i))+\\ &\frac{2ck^{2}J^{2}}{t}\sum_{j\sim i}\nu(f^{2}d^2(\omega_j))
\end{align*}
For $J$ sufficiently  small so that $\frac{8ck^{2}J^{2}}{t}<1$ and $$\frac{2ck^{2}J^{2}}{t}+\frac{24c^{2}k^{4}J^{4}}{t}<\frac{3ck^{2}J^{2}}{t}\Rightarrow  \frac{\frac{2ck^{2}J^{2}}{t}}{1-\frac{8ck^{2}J^{2}}{t}}<\frac{3ck^{2}J^{2}}{t}$$
we obtain
\begin{align*}\nu ( d^2(x_i)f^2)\leq\frac{K}{(1-\frac{8ck^{2}J^{2}}{t})t}\nu( f^{2})+\frac{2c}{(1-\frac{8ck^{2}J^{2}}{t})t}\nu(\|\nabla_{i} f\|^2)+\frac{3ck^{2}J^{2}}{t} \sum_{j\sim i}\nu(f^{2}d^2(\omega_j))
\end{align*}and the lemma follows for appropriate constants $K_0$ and $K_1$.
\end{proof}
In the next lemma we show a technical calculation of an iteration that will be used.\begin{lem}\label{lemU2}If for any $i\in  \mathbb{Z}^2$ 
\begin{align}\label{3eq5}  P(i)\leq R(i)+bs^{2}\sum_{j\sim i}P(j)
\end{align}for some    $b>0$ and  some $s\in(0,1)$ sufficiently small, and
\begin{align}\label{3eq6}\lim_{n\rightarrow \infty}(s^n\sum _{dist(j,k)=n}P(k))=0 \ \  \   \  \   \forall k \in  \mathbb{Z}^2 
\end{align}
then
  \begin{align*}  P(k)\leq\frac{1}{1-bs}\sum_{n=0}^{\infty}\left(s^n\sum_{j:dist(j,k)=n}\   R(j)   \right)
\end{align*}for any  $k\in  \mathbb{Z}^2$. 
\end{lem}\begin{proof}
We will first  show that for any  $n\in\mathbb{N}$ there exists an $s\in(0,1)$ such that  \begin{align}\label{3eq7}  \sum _{dist(j,k)=n}P(j)\leq  &\frac{1}{1-3bs^2}\sum _{dist(j,k)=n}R(j)+\sum_{t=0}^{n-1}s^{n-t}\sum _{dist(j,k)=t}R(j)+s \sum _{dist(j,k)=n+1}P(j) \end{align}We will work by induction.  

Step 1: The base step of the induction ($n=1$).
 From
 (\ref{3eq5}) we have  \begin{align*} \sum _{dist(t,k)=1}P(t) \leq &  \sum _{dist(t,k)=1}R(t)+bs^{2}\sum _{dist(t,k)=1}\sum_{i\sim t}P(i)\leq\sum _{dist(t,k)=1}R(t)+\\& +2bs^{2}\sum _{dist(t,k)=2} P(t)+4bs^{2}  P(k)  \end{align*}If we use again (\ref{3eq5}) to  bound the last term we obtain   \begin{align*} \sum _{dist(t,k)=1}P(t) \leq&\sum _{dist(t,k)=1}R(t)+2bs^{2}\sum _{dist(t,k)=2} P(t)\\& +4b^{2}s^{4}\sum _{dist(t,k)=1} P(t)+4bs^{2}  R(k)  \end{align*}For $s$ small enough so that $4b^{2}s^{4}<1$, $4bs^{2} \leq 3 $ and $4bs^{2}+4b^{2}s^{5}<s$ we have   \begin{align*} \sum _{dist(t,k)=1}P(t) \leq&\frac{1}{1-4b^{2}s^{4}}\sum _{dist(t,k)=1}R(t)+\frac{2bs^{2}}{1-4b^{2}s^{4}}\sum _{dist(t,k)=2} P(t)\\& +\frac{4bs^{2}}{1-4b^{2}s^{4}}R(k)  \\ \leq&\frac{1}{1-3bs^{2}}\sum _{dist(t,k)=1}R(t)+s\sum _{dist(t,k)=2} P(t)\\& +sR(k)\end{align*}since $4bs^{2} \leq 3\Rightarrow\frac{1}{1-4b^{2}s^{4}}\leq\frac{1}{1-3bs^{2}}$ and $4bs^{2}+4b^{2}s^{5}<s\Rightarrow\frac{4bs^2}{1-4b^{2}s^{4}}\leq s$. This proves the base step.  
 
Step 2: The induction step. We assume that (\ref{3eq7}) holds true for some $n\geq2$, and we will show that it also holds for $n+1$, that is 
\begin{align}\label{3eq8}\sum _{dist(j,k)=n+1}P(j)\leq &  \frac{1}{1-3bs^2}\sum _{dist(j,k)=n+1}R(j)+  \sum_{t=0}^{n}s^{n+1-t}\sum _{dist(j,k)=t}R(j)+\\ 
&s \sum _{dist(j,k)=n+2}P(j)\nonumber\end{align}
 To bound the left hand side of (\ref{3eq8}) we can use again (\ref{3eq5})  \begin{align*} \sum _{dist(j,k)=n+1}P(j)\leq & \sum _{dist(j,k)=n+1}R(j)+bs^{2}\sum _{dist(j,k)=n+1}\sum_{t\sim j}P(t)\leq\ \\  & \sum _{dist(j,k)=n+1}R(j)+3bs^{2}\sum _{dist(j,k)=n} P(j)+2bs^{2}\sum _{dist(j,k)=n+2} P(j)\end{align*}
If we bound $\sum _{dist(j,k)=n} P(j)$ by (\ref{3eq7}) we get 
 \begin{align*} \sum _{dist(j,k)=n+1}P(j)\leq &   \sum _{dist(j,k)=n+1}R(j)+\frac{3bs^{2}}{1-3bs^2}\sum _{dist(j,k)=n}R(j)+ \\& 3bs^{2}\sum_{t=0}^{n-1}s^{n-t}\sum _{dist(j,k)=t}R(j)+\\  & 3bs^{3}\sum _{dist(j,k)=n+1} P(j)+2bs^{2}\sum _{dist(j,k)=n+2} P(j)\end{align*}
For $s$ small enough such that $3bs^2<1$, $\frac{3bs^2}{1-3bs^3}\leq s$ and $3b(s^{2}+s^{3}+s^{4})\leq 9b^2s^6+s\Rightarrow\frac{3bs^{2}}{(1-3bs^3)(1-3bs^2)}\leq s$ we obtain 
 \begin{align*} \sum _{dist(j,k)=n+1}P(j)\leq &   \frac{1}{1-3bs^2} \sum _{dist(j,k)=n+1}R(j)+  \sum_{t=0}^{n}s^{n+1-t}\sum _{dist(j,k)=t}R(j)+\\  &  +s\sum _{dist(j,k)=n+2} P(t)\end{align*}
which finishes the proof of (\ref{3eq7}). 

We can now complete the proof of the lemma. At first we can bound the second term on the right hand side of (\ref{3eq5}) by (\ref{3eq7}). That gives \begin{align*} P(k)\leq & (1+bs^3) R(k)+\frac{bs^{2}}{1-3bs^{2} }\sum_{j:dist(j,k)=1}R(j)+bs^{3}\sum_{dist(j,k)=2}P(j)\leq \\  & (1+bs^3)R(k)+s\sum_{j:dist(j,k)=1}R(j)+bs^{3}\sum_{dist(j,k)=2}P(j)
\end{align*}for $s$ sufficiently small so that $\frac{bs^{2}}{1-3bs^{2} }\leq s$. If we use again (\ref{3eq7}) to bound the third term on the right hand we have  \begin{align*} P(k)\leq  & (1+bs^3)R(k)+s\sum_{j:dist(j,k)=1}R(j)+s^{2}\sum _{dist(j,k)=2}R(j)+\\ &b\sum_{t=0}^{1}s^{5-t}\sum _{dist(j,k)=t}R(j)+bs^{4} \sum _{dist(j,k)=3}P(j)
\end{align*} where above we used once more that $\frac{bs^{2}}{1-3bs^{2} }\leq s$. If we rearrange the terms we have  \begin{align*} P(k)\leq(1+bs^3)  & R(k)+(s+bs^{4})\sum_{j:dist(j,k)=1}R(j)+s^{2}\sum _{dist(j,k)=2}R(j)+\\ & bs^{4} \sum _{dist(j,k)=3}P(j)
\end{align*}  If we continue inductively to bound the right hand side by (\ref{3eq7}) and take under account (\ref{3eq6}), then for $s$ sufficiently small such that $sb<1$ we obtain 
  \begin{align*} P(k)\leq(\sum_{r=0}^{\infty}(bs)^r)\sum_{n=0}^{\infty}\left(s^n\sum_{j:dist(j,k)=n}\   R(j)   \right)
\end{align*}   which proves the lemma.
  \end{proof}
  We now prove the U-bound inequality of Proposition \ref{propUbound}.

  \subsection{\underline{proof of Proposition \ref{propUbound}}.}
\begin{proof}The proof of the proposition follows directly from Lemma \ref{lemU1} and Lemma \ref{lemU2}. If one considers $P(k):=\nu ( d^2(x_k)f^2)$ and $R(k):=K_{0}\nu( f^{2})+K_{0}\nu(\|\nabla_{k} f\|^2)$ then  from (\ref{induction}) of  Lemma \ref{lemU1} we see that condition (\ref{3eq5}) is satisfied for $b:=K_{1}=\frac{3ck^{2}}{t}$ and $s=J$. 

Furthermore, since by our hypothesis   $\nu ( d^2(x_i)f^2)\leq M<\infty$ for some positive $M$ uniformly on $i$ and  $\{\# j:dist(j,k)=n\}\leq4^n$, we can choose $J$ sufficiently small so that  for $s=J<\frac{1}{4}$  condition (\ref{3eq6}) of Lemma \ref{lemU2} to be also satisfied:
\begin{align*}\lim_{n\rightarrow \infty}(s^n\sum _{dist(j,k)=n}P(k))\leq M\lim_{n\rightarrow \infty}(4s)^{n}=0 \ \  \   \  \   \forall k \in  \mathbb{Z}^2
\end{align*}
Since (\ref{3eq5}) and (\ref{3eq6}) are satisfied we can apply  Lemma \ref{lemU2}. We then obtain \begin{align*}  \nu ( d^2(x_k)f^2)\leq &\frac{1}{1-K_{1}J}\sum_{n=0}^{\infty}\left(J^{n}\sum_{j:dist(j,k)=n}\   (K_{0}\nu( f^{2})+K_{0}\nu(\|\nabla_{j} f\|^2)   )\right)\leq \\  &  \frac{K_{0}}{1-K_{1}J}\left(\sum_{n=0}^{\infty}(4J)^{n}\    \nu( f^{2})+ \sum_{n=0}^{\infty} J^n\sum_{j:dist(j,k)=n}\   \nu(\|\nabla_{j} f\|^2   )\right)
\end{align*}For $J$ small enough so that $4J<1$ we get
 \begin{align*}  \nu ( d^2(x_k)f^2)\leq\   \frac{K_{0}}{(1-K_{1}J)(1-4J)}  \nu( f^{2})+\frac{K_{0}}{1-K_{1}J} \sum_{n=0}^{\infty} J^n\sum_{j:dist(j,k)=n}\   \nu(\| \nabla_{j} f\|^2   )
\end{align*}which proves the proposition.
\end{proof}
 \section{\label{section4}First Sweeping Out Inequality.}Sweeping out inequalities for the local specification were introduced in  \cite{Z1},  \cite{Z2} and \cite{G-Z} to prove logarithmic Sobolev inequalities. Here we prove a  weaker version of them for the Gibbs measure, similar to the ones used in \cite{Pa1} and \cite{Pa3}, where however, interactions higher than  quadratic  were considered.  
 \begin{lem} \label{4lem1} Assume   that the  measure $\mu $ satisfies the log-Sobolev inequality and  that the  local specification has quadratic interactions $V$ as in (\ref{quadratic}). Then, for $J$ sufficiently small,  for every $j \sim i$   \begin{align*}
  \nu\| \nabla_{j}(\mathbb{E}^{i}f)
\|^2\leq & D_{1} \nu\|\nabla_{j}f\|^{2}+ D_{1}J^{2}\nu\mathbb{E}^{i}(f-\mathbb{E}^{i}f)^{2} \ +  \\ & J D_{1}      \sum_{n=1}^{\infty} J^{n}\sum_{ dist(r,i) =n}\   \nu(\|\nabla_{r} f\|^2   )\end{align*}  for some constant $D_1\in [1,\infty)$.  \end{lem}

 \begin{proof}  Consider the (sub)gradient $\nabla_j=(X_1^j,X_2^j,...,X_N^j)$. We can then write
 \begin{align} \label{4eq1} \| \nabla_j(\mathbb{E}^i f)
\|^2=\sum_{k=1}^N( X_k^j(\mathbb{E}^i f))^{2}\end{align}
If we denote   $\rho_i= \frac{e^{-H(x_{i})}}{\int e^{-H(x_{i})}dx}$ the density of the measure $\mathbb{E}^{i}$, then for every $k=1,...,N$ we have
\begin{align}\nonumber ( X^k_j(\mathbb{E}^i f))^{2}\leq\ & \left\vert X^k_j(\int \rho_i  f dx_i)\right\vert^2\leq
  & \\ &  \label{4eq2}2\left\vert\int(X^k_jf) \rho_i dx_i\right\vert^2+ 2 \left\vert\int \int f(X^k_j\rho_i )dx_i\right\vert^2\leq  \end{align} \begin{align}\label{4eq3}2\left\vert\mathbb{E}^{i}(X^k_jf)\right\vert^{2}+ 2J^2 \left\vert\mathbb{E}^i(f; X^k_jV(x_j,x_i))\right\vert^2
\end{align}
where in~\eqref{4eq3}  we  bounded the coefficients $J_{i,j}$
by $J$ and we have denoted $\mathbb{E}^i(f;g)$ the covariance of $f$ and $g$. If we take expectations with respect to the Gibbs measure $\nu$ and use the H\"older  inequality  in both terms  of~\eqref{4eq3}  we obtain
\begin{align*}\nonumber\nu \left\vert X^k_j(\mathbb{E}^{i}f)
\right\vert^2\leq   \nonumber 2\nu (X^k_jf)^{2}+\ 2J^2\nu\mathbb{E}^{i}((f-\mathbb{E}^{i}f)^{2} (X^k_jV(x_{j},x_{i}))^{2})\end{align*}If we take the sum over all $k$ from $1$ to $N$ in the last inequality and take under account (\ref{4eq1}) we get
\begin{align*}\| \nabla_j(\mathbb{E}^i f)
\|^2\leq &2\nu\| \nabla_jf\|^{2}+\ 2J^2\nu\mathbb{E}^{i}((f-\mathbb{E}^{i}f)^{2} \|\nabla_jV(x_{j},x_{i})\|^{2}) \\   \leq & \nonumber 2\nu \|\nabla_{j}f\|^{2}+ k2J^2\nu\mathbb{E}^{i}((f-\mathbb{E}^{i}f)^{2} (d^{2}(x_{j})+d^{2}(x_{i}))\end{align*}
where above we used that  the interactions are quadratic as in  hypothesis (\ref{quadratic}).  This leads to \begin{align}\nonumber \label{4eq4}\nu \| \nabla_{j}(\mathbb{E}^{i}f)\| ^2\leq & 2\nu \|\nabla_{j}f\|^{2}+k2J^2\nu (f-\mathbb{E}^{i}f)^{2} \ +  \\ & k2J^2\nu ((f-\mathbb{E}^{i}f)^{2} d^{2}(x_{i}))+k2J^2\nu ((f-\mathbb{E}^{i}f)^{2}d^{2}(x_{j}))\end{align} In order to bound the third term on the  right  hand side of (\ref{4eq4}) we can use  Proposition \ref{propUbound} \begin{align*}  \nu ((f-\mathbb{E}^{i}f)^{2} d^{2}(x_{i}))\leq & C\nu (f-\mathbb{E}^{i}f)^{2}+C \sum_{n=0}^{\infty} J^{n}\sum_{k:dist(k,i)=n}\   \nu(\|\nabla_{k} (f-\mathbb{E}^{i}f)\|^2   )   \leq \\ & \nonumber  C\nu (f-\mathbb{E}^{i}f)^{2}+2C \sum_{n=0}^{\infty} J^{n}\sum_{k:dist(k,i)=n}\   \nu(\|\nabla_{k} f\|^2   )+\\ &2C \sum_{n=0}^{\infty} J^{n}\sum_{k:dist(k,i)=n}\   \nu(\|\nabla_{k} (\mathbb{E}^{i}f)\|^2   ) \nonumber  \end{align*}
Since $\nabla_{k} (\mathbb{E}^{i}f)=\mathbb{E}^{i}(\nabla_{k} f)$ when $dist(k,i)>1$ and $\nabla_{i} (\mathbb{E}^{i}f)=0$ the last inequality takes the form
\begin{align}\label{4eq5}  \nu ((f-\mathbb{E}^{i}f)^{2} d^{2}(x_{i}))\leq  &   C\nu (f-\mathbb{E}^{i}f)^{2}+2C \sum_{n=0}^{\infty} J^{n}\sum_{r:dist(r,i)=n}\   \nu(\|\nabla_{r} f\|^2   )+\\ &2C  J\sum_{j\sim i}\   \nu(\|\nabla_{j} (\mathbb{E}^{i}f)\|^2   ) \nonumber  \end{align}
    For the fourth term on the  right  hand side of (\ref{4eq4}) we can use  again Proposition \ref{propUbound} \begin{align}\label{4eq6}\nonumber \nu ((f-\mathbb{E}^{i}f)^{2} d^{2}(x_{j}))\leq & C\nu (f-\mathbb{E}^{i}f)^{2}+C \sum_{n=0}^{\infty} J^{n}\sum_{r:dist(r,j)=n}\nu \|  \nabla_{r}(f-\mathbb{E}^{i}f)\|^{2}\leq  \\  \nonumber & C\nu (f-\mathbb{E}^{i}f)^{2}+2C \sum_{n=0}^{\infty} J^{n}\sum_{r:dist(r,j)=n}\nu \|  \nabla_{r}f\|^{2} +\\ &2C \sum_{n=0}^{\infty} J^{n}\sum_{r:dist(r,j)=n}\nu\|  \nabla_{r}( \mathbb{E}^{i}f)\|^{2}  \end{align}
  But   $\nabla_{r} (\mathbb{E}^{i}f)=\mathbb{E}^{i}(\nabla_{r} f)$ when $dist(r,i)>1$ and $\nabla_{i} (\mathbb{E}^{i}f)=0$. Furthermore, since $j \sim i$ when $ dist(r,j)=n$, the $r$'s that neighbour $i$ will have distance from  $j$ equal to    $2$ when $r\neq j$ or $0$ when $r=j$. 
 So (\ref{4eq6}) becomes  \begin{align}\label{4eq7}\nonumber  \nu ((f-\mathbb{E}^{i}f)^{2} d^{2}(x_{j}))\leq    \nonumber & C\nu\mathbb{E}^{i}(f-\mathbb{E}^{i}f)^{2}+2C \sum_{n=0}^{\infty} J^{n}\sum_{r:dist(r,j)=n}\nu \|  \nabla_{r}f\|^{2} +\\ &2C\nu \|  \nabla_{j}( \mathbb{E}^{i}f)\|^{2}+2C   J^{2}\sum_{r\sim i:r\neq j }\nu \|  \nabla_{r}( \mathbb{E}^{i}f)\|^{2}  \end{align}
 If we combine  together  (\ref{4eq4}),  (\ref{4eq5}) and (\ref{4eq7}) we get \begin{align}\label{4eq8} \nonumber \nu \| \nabla_{j}(\mathbb{E}^{i}f)
\|^2\leq & (2+16kCJ^2)\nu\|\nabla_{j}f\|^{2}+(1+2C)2k J^2\nu (f-\mathbb{E}^{i}f)^{2} \ +  \\ &  8kCJ      \sum_{n=0}^{\infty} J^{n}\sum_{r:dist(r,i)=n}\   \nu\|\nabla_{r} f\|^2   +8kJ^2C  \sum_{r\sim i}\   \nu(\|\nabla_{r} (\mathbb{E}^{i}f)\|^2   ) \end{align}since $J^2<J$ and 
  \begin{align}\label{newJnew}      \sum_{n=0}^{\infty} J^{n+1}\sum_{r:dist(r,j)=n}\   \nu\|\nabla_{r} f\|^2 \leq       \sum_{n=0}^{\infty} J^{n}\sum_{r:dist(r,i)=n}\   \nu\|\nabla_{r} f\|^2\end{align}
 for every $i \sim j$. If we take the sum over all $j\sim i$ in both sides of the inequality we will obtain 
 \begin{align*}\nonumber  \sum_{j\sim i}\nu\| \nabla_{j}(\mathbb{E}^{i}f)\|^2\leq & (2+16kCJ^2)\sum_{j\sim i}\nu\|\nabla_{j}f\|^{2}+(1+2C)8kJ^2\nu (f-\mathbb{E}^{i}f)^{2} \ +  \\ & 32kCJ      \sum_{n=0}^{\infty} J^{n}\sum_{r:dist(r,i)=n}\   \nu\|\nabla_{r} f\|^2  +32kJ^2C  \sum_{r\sim i}\   \nu\|\nabla_{r} (\mathbb{E}^{i}f)\|^2  \end{align*}If we choose $J$ sufficiently  small  so that $$\frac{(1+ 4C)8kJ}{1-32kJ^2C  }\leq  (1+ 4C)16kJ\Leftrightarrow\ J \leq \frac{1}{32kC} $$ we get 
 \begin{align*}\nonumber  \sum_{j\sim i}\nu\| \nabla_{j}(\mathbb{E}^{i}f)
\|^2\leq & \frac{(2+16kCJ)}{1-32kCJ^2} \sum_{j\sim i}\nu\|\nabla_{j}f\|^{2}+(1+ 4C)16kJ^2 \nu (f-\mathbb{E}^{i}f)^{2} \ +  \\ & (1+ 4C)16kJ     \sum_{n=0}^{\infty} J^{n}\sum_{r:dist(r,i)=n}\   \nu\|\nabla_{r} f\|^2  \end{align*} 
  Plugging the last one  into (\ref{4eq8}) and choosing $J$ small enough so that 
  $$\frac{(2+ 16kCJ^2)C8kJ}{1-32kCJ^2  }\leq   32CkJ\Leftrightarrow\ J \leq \frac{ 1}{72kC} $$ we obtain
  \begin{align*} \nonumber \nu \| \nabla_{j}(\mathbb{E}^{i}f)
\|^2\leq & (2+16kCJ^2)\nu\|\nabla_{j}f\|^{2}+(1+4C)128k^{2}C J^2\nu\mathbb{E}^{i}(f-\mathbb{E}^{i}f)^{2} \ +  \\ & \nonumber 32CkJ^2 \sum_{r\sim i}\nu\|\nabla_{r}f\|^{2}+    (1+ 4C)136k^2CJ    \sum_{n=0}^{\infty} J^{n}\sum_{r:dist(r,i)=n}\   \nu\|\nabla_{r} f\|^2    \end{align*}which finishes the proof for appropriate chosen constant $D_1$.
  \end{proof}
  Furthermore combining together (\ref{4eq7}) and  Lemma \ref{4lem1}, we obtain  the following  corollary.
  \begin{cor}\label{4cor2}   Assume   that the  measure $\mu $ satisfies the log-Sobolev inequality and  that the  local specification has quadratic interactions $V$ as in (\ref{quadratic}). Then, for $J$ sufficiently small,   for every $j \sim i$ the following holds
\begin{align*}  \nu ((f-\mathbb{E}^{i}f)^{2} d^{2}(x_{j}))\leq &D_{2}\nu\mathbb{E}^{i}(f-\mathbb{E}^{i}f)^{2}+ D_{2} \nu \|\nabla_{j}f\|^{2} +\\ & D_{2}      \sum_{n=0}^{\infty} J^{n}\sum_{ dist(r,j) =n}\   \nu\|\nabla_{r} f\|^2    \end{align*} for some constant $D_2>0$.
  \end{cor}
  where in the above  corollary we used again (\ref{newJnew}). The next lemma shows  the Poincar\'e inequality for the one site   measure $\mathbb{E}^{ i}$ on the ball.
 The proof follows closely on the proof of a similar Poincar\'e inequality on  the ball in  \cite{I-K-P} and the local Poincar\'e inequalities from \cite{SC} and \cite{V-SC-C}. \begin{lem}\label{PoincareBall} Define $\eta(i,\omega):= d( x_i ) +\sum_{j\sim i}d( \omega_{j})$ and $A^\omega(L):=\{x_i\in M:\eta(i,\omega)\leq L\}$. For any $L>0$ the following Poincar\'e type inequality on the ball holds$$\mathbb{E}^{i,\omega} \vert f-\frac{1}{\vert A^{\omega}(L)\vert}\int_{A^{\omega}(L)} f(z_i)dz_i\vert^2\mathcal{I}_{\{\eta(i,\omega) \leq L\}}  \nonumber\vert^2\mathcal{I}_{\{\eta(i,\omega) \leq L\}}\leqslant D_{L}\mathbb{E}^{i,\omega}\| \nabla_{i} f\|^2$$
for some positive constant $D_L$, where    $\vert A^{\omega}(L)\vert :=\int_{A^{\omega}(L)}dx_i$.\end{lem}
\begin{proof} Denote
 \begin{align}\nonumber  V_L:= \mathbb{E}^{i} \vert f -\frac{1}{\vert A^{\omega}(L)\vert}\int_{A^{\omega}(L)} f(z_i)dz_i\vert^2\mathcal{I}_{\{\eta(i,\omega) \leq L\}}  \nonumber \end{align}
where $\mathbb{E}^{i} $ has density $\rho_i=\frac{e^{-H^{ i,\omega}}}{\int e^{-H^{i,\omega}dx_i}dx_{i} }$. Since   $\phi( x_{i} )\geq0$ and  $J_{i,j}V( x_{i},\omega_{j})\geq 0$ we can bound $\rho_i\leq \frac{1}{Z^{i,\omega}}$. This leads to 
\begin{align}\label{4eq10}V_{L}\leq\frac{1}{Z^{i,\omega}}  \int_{A^{\omega}(L)} \left\vert f(x_{i})-\frac{1}{\vert A^{\omega}(L)\vert}\int_{A^{\omega}(L)} f( z_i)dz_i\right\vert^2 dx_{i} 
\end{align} If we use the invariance of the $dx_i$ measure we can write  
\begin{align*}V_{L}&\leq\frac{1}{Z^{i,\omega}}  \int_{A^{\omega}(L)} \vert f(x_{i})-\frac{1}{\vert A^{\omega}(L)\vert}\int_{A^{\omega}(L)} f(x_iz_{i})\mathbb{\mathcal{I}}_{A^{\omega}(L)}(x_iz_i)dz_i\vert^2 dx_{i}\\ &\nonumber\leq \frac{1}{\vert A^{\omega}(L)\vert^{2} Z^{i,\omega}} \int_{A^{\omega}(L)} \vert\int_{A^{\omega}(L)}  f(x_{i})- f(x_iz_{i})\mathbb{\mathcal{I}}_{A^{\omega}(L)}(x_iz_i)dz_i\vert^2 dx_{i} \end{align*}
If we use Holder inequality and consider $L$ sufficiently large so that $\vert A^{\omega}(L)\vert>1$   \begin{align}\label{4eq11} V_{L}&\leq\ \frac{1}{\vert A^{\omega}(L)\vert Z^{i,\omega}} \int \int \vert f(x_{i})- f(x_iz_{i})\vert^2 \mathbb{\mathcal{I}}_{A^{\omega}(L)}(x_iz_i)\mathbb{\mathcal{I}}_{A^{\omega}(L)}(x_i )dz_idx_{i}
\end{align}
Consider $\gamma:[0,t]\rightarrow M$ a geodesic from $0$ to $z_i$ such that $\vert \dot \gamma(t) \vert \leq 1$. Then for $t=d(z_i)$ we can write 
\begin{align*}\vert f(x_{i})- f(x_iz_{i})\vert^2 =&\vert\int_{0}^t \frac{d}{ds} f(x_{i}\gamma(s))ds\vert^2=\vert\int_{0}^t \nabla_{i}\ f(x_{i}\gamma(s))\cdot \dot\gamma(s)ds\vert^2 \leq \\  &t\int_{0}^t\| \nabla_{i}\ f(x_{i}\gamma(s))\|^2 ds=d(z_i)\int_{0}^t\| \nabla_{i}\ f(x_{i}\gamma(s))\|^2 ds \end{align*}
From the last inequality and (\ref{4eq11}) we get   \begin{align*} V_{L}&\leq\ \frac{1}{\vert A^{\omega}(L)\vert Z^{i,\omega}} \int \int d(z_i)\int_{0}^t\| \nabla_{i}\ f(x_{i}\gamma(s))\|^2 ds \mathbb{\mathcal{I}}_{A^{\omega}(L)}(x_iz_i)\mathbb{\mathcal{I}}_{A^{\omega}(L)}(x_i )dz_idx_{i}
\end{align*}
We observe  that for $x_i \in A^{\omega}(L)$ and $x_i z_i\in A^{\omega}(L)$ we obtain $$d(z_i)=d(x_i^{-1}x_iz_i)\leq d(x_i^{-1} )+d( x_iz_i)\leq d(x_{i})+d( x_iz_i)\leq2L\ $$
So  \begin{align*} V_{L}&\leq\ \frac{2L}{\vert A^{\omega}(L)\vert Z^{i,\omega}} \int \int \int_{0}^t\| \nabla_{i}\ f(x_{i}\gamma(s))\|^2 ds \mathbb{\mathcal{I}}_{A^{\omega}(L)}(x_iz_i)\mathbb{\mathcal{I}}_{A^{\omega}(L)}(x_i )dz_idx_{i}
\end{align*}
Similarly,    for $x_i \in A^{\omega}(L)$ and $x_i z_i\in A^{\omega}(L)$ we calculate 
$$\eta(i,\omega)( z_i)=d( z_i ) +\sum_{j\sim i}d( \omega_{j})\leq d(x_{i})+d( x_iz_i)+\sum_{j\sim i}d( \omega_{j})\leq2L$$
as well  as
\begin{align*}\eta(i,\omega)( x_i \gamma(s))=&d( x_i \gamma(s)) +\sum_{j\sim i}d( \omega_{j})\leq d( \gamma(s)) +d( x_i  )+\sum_{j\sim i}d( \omega_{j})\leq \\  &  d(z_{i}) +d( x_i  )+\sum_{j\sim i}d( \omega_{j})\leq3 L
\end{align*}So, we can write 
\begin{align*} V_{L}&\leq\ \frac{2L}{\vert A^{\omega}(L)\vert Z^{i,\omega}} \int \int \int_{0}^t\| \nabla_{i}\ f(x_{i}\gamma(s))\|^2  \mathbb{\mathcal{I}}_{A^{\omega}(3L)}( x_i \gamma(s))\mathbb{\mathcal{I}}_{A^{\omega}(2L)}(z_i )dsdz_idx_{i}
\end{align*}
Using again the invariance of the  $dx_i$ measure
\begin{align*} V_{L}&\leq  \frac{2L }{\vert A^{\omega}(L)\vert Z^{i,\omega}} \int \int \int_{0}^t\| \nabla_{i}\ f(x_{i})\|^2   \mathbb{\mathcal{I}}_{A^{\omega}(3L)}(x_{i})\mathbb{\mathcal{I}}_{A^{\omega}(2L)}(z_i )dsdx_{i}dz_i\\  &  =\frac{2L }{\vert A^{\omega}(L)\vert Z^{i,\omega}} \int \int  d(z_{i})\| \nabla_{i}\ f(x_{i})\|^2  \mathbb{\mathcal{I}}_{A^{\omega}(3L)}(x_{i})\mathbb{\mathcal{I}}_{A^{\omega}(2L)}(z_i )dx_{i}dz_i  \\ & \leq \frac{4L^{2}}{\vert A^{\omega}(L)\vert Z^{i,\omega}} \int \int   \left(\| \nabla_{i}\ f(x_{i})\|^2  \mathbb{\mathcal{I}}_{A^{\omega}(3L)}(x_{i})\right)dx_{i}\mathbb{\mathcal{I}}_{A^{\omega}(2L)}(z_i )dz_i \\
 & \leq \frac{4L^{2} \vert A^{\omega}(2L)\vert}{\vert A^{\omega}(L)\vert Z^{i,\omega}} \int \| \nabla_{i}\ f(x_{i})\|^2   \mathbb{\mathcal{I}}_{A^{\omega}(3L)}(x_{i})dx_{i} 
\end{align*}
For $x_i \in A^{\omega}(3L)$, since $d^{2}(x_i)\leq 9L^2$ and $d^{2}(\omega_j)\leq9L^2$  the Hamiltonian is bounded by  
\begin{align*}H^{i,\omega}=&\phi(x_i)+\sum_{j \sim i}J_{i,j}V(x_i,\omega_j)\leq \sup_{\{d(x)\leq 3L\}}\phi(x) +4J+4Jkd^{2}(x_{i})+Jk\sum_{j \sim i}d^{2}(\omega_j) \\  \leq  & \sup_{\{d(x)\leq 3L\}}\phi(x)+72JkL^{2}+4J:=F_{L} \ \end{align*}
So $$e^{-H^{i,\omega}}\geq e^{-F_{L} } $$
which gives the following bound 
\begin{align*} V_{L}&\leq \frac{4L^{2} }{e^{F_{L} }}    \frac{ \vert A^{\omega}(2L)\vert}{\vert A^{\omega}(L)\vert Z^{i,\omega}} \int \| \nabla_{i}\ f(x_{i})\|^2   e^{-H^{i,\omega}}dx_{i}  \\  & =\frac{4L^{2} }{e^{F_{L} }}    \frac{ \vert A^{\omega}(2L)\|}{\vert A^{\omega}(L)\vert } \mathbb{E}^{i,\omega}\| \nabla_{i}\ f(x_{i})\|^2   
\end{align*}
If we take under account that $$\frac{ \vert A^{\omega}(2L)\vert}{\vert A^{\omega}(L)\vert }\geq 1$$as well as
$$\frac{ \vert A^{\omega}(2L)\vert}{\vert A^{\omega}(L)\vert }\rightarrow1\ \ \text{as} \  \  \sum_{j\sim i}d(\omega_j)\rightarrow \infty$$
we observe that $\frac{ \vert A^{\omega}(2L)\vert}{\vert A^{\omega}(L)\vert }$ is bounded from above uniformly on $\omega$ from a constant. Thus, we finally obtain that
\begin{align*} V_{L}&\leq  D_{L} \mathbb{E}^{i,\omega}\| \nabla_{i}\ f(x_{i})\|^2   \end{align*}
for some positive  constant $D_L$.

 \end{proof}The next lemma gives a   bound for  the variance of the one site measure $\mathbb{E}^{i}$ outside $A^{\omega}(L)$.

\begin{lem}\label{PoincareNOBall}Assume   that the  measure $\mu $ satisfies the log-Sobolev inequality and  that the  local specification has quadratic interactions $V$ as in (\ref{quadratic}). Then, for $J$ sufficiently small the following bound holds \begin{align*}\nu\vert f- m\vert^2\mathbb{\mathcal{I}}_{\{\eta(i,\omega) > L\}} \nonumber \leq D_{3}\nu(\|\nabla_{i} f \|^2   )+D_{3} \sum_{n=1}^{\infty} J^{n-1}\sum_{r:dist(r,i)=n}\   \nu(\|\nabla_{r} f \|^2   )+ \end{align*}
for any   $L>5+2C$ and $\forall m \in \mathbb{R}$.\end{lem}
\begin{proof}
 We  can write
\begin{align*}\nonumber   \mathbb{E}^{i}\left(\vert f- m\vert^2\mathbb{\mathcal{I}}_{\{\eta(i,\omega) > L\}}\right)&\leq \mathbb{E}^{i} \left(\vert f-m\vert^2\frac{\eta(i,\omega) }{L}\mathbb{\mathcal{I}}_{\{\eta(i,\omega) > L\}}\right)\\  &= \frac{1}{L}\mathbb{E}^{i}\left(\vert f-m\vert^2\left(d( x_{i} )+ \sum_{j\sim i}d(\omega_{j} )  \right)\mathbb{\mathcal{I}}_{\{\eta(i,\omega) > L\}}\right)\nonumber \\ \nonumber &\leq \frac{1}{L}\mathbb{E}^{i}\left(\vert f-m\vert^2\left(d^{2}( x_{i} )+ \sum_{j\sim i}d^{2}(\omega_{j} )+5\right)\mathbb{\mathcal{I}}_{\{\eta(i,\omega) > L\}}\right)\end{align*}
 since $d(x)=d(x)\mathbb{\mathcal{I}}_{\{d(x) \leq1\}}+d(x)\mathbb{\mathcal{I}}_{\{d(x) >1\}}\leq 1+d^2(x)$. If we take the expectation with respect to the Gibbs measure we get 
\begin{align*}\nonumber  \nu\left(\vert f- m\vert^2\mathbb{\mathcal{I}}_{\{\eta(i,\omega) > L\}}\right) \nonumber \leq &\frac{1}{L}\nu\left(\vert f-m\vert^2 d^{2}( x_{i} )\mathbb{\mathcal{I}}_{\{\eta(i,\omega) > L\}}\right)+\\ & \nonumber\frac{1}{L}\sum_{j\sim i}\nu\left(\vert f-m\vert^2 d^{2}( \omega_{j} )\mathbb{\mathcal{I}}_{\{\eta(i,\omega) > L\}}\right)+\frac{5}{L}\nu(\vert f-m\vert^2  \mathbb{\mathcal{I}}_{\{\eta(i,\omega) > L\}})\end{align*}
  We can bound  the first and second term  on the right hand side  from Proposition \ref{propUbound}
to get  
\begin{align*}\nonumber  \nu\vert f- m\vert^2\mathbb{\mathcal{I}}_{\{\eta(i,\omega) > L\}} \nonumber \leq & \frac{C}{L} \sum_{n=0}^{\infty} J^{n}\sum_{r:dist(r,i)=n}\   \nu(\|\nabla_{r} f \|^2   )+\\ &  \frac{C}{L}\sum_{j\sim i} \sum_{n=0}^{\infty} J^{n}\sum_{r:dist(r,j)=n}\   \nu(\|\nabla_{r} f \|^2   )+\\  & \nonumber \frac{5+5C}{L}\nu(\vert f-m\vert^2  \mathbb{\mathcal{I}}_{\{\eta(i,\omega) > L\}})\end{align*}
  Which leads to \begin{align*}  \nu\vert f- m\vert^2\mathbb{\mathcal{I}}_{\{\eta(i,\omega) > L\}} \nonumber \leq & \frac{5C}{L}  \   \nu(\|\nabla_{i} f \|^2   )+\\ &  \frac{5C}{L} \sum_{n=1}^{\infty} J^{n-1}\sum_{r:dist(r,i)=n}\   \nu(\|\nabla_{r} f \|^2   )+\\  & \nonumber \frac{5+5C}{L}\nu(\vert f-m\vert^2  \mathbb{\mathcal{I}}_{\{\eta(i,\omega) > L\}})
  \end{align*}   If we choose $L$ sufficiently large so that $\frac{5+5C}{L}<1$ we finally obtain\begin{align*}  \nu\vert f- m\vert^2\mathbb{\mathcal{I}}_{\{\eta(i,\omega) > L\}} \nonumber \leq    D_{3}\nu(\|\nabla_{i} f \|^2   )+D_{3} \sum_{n=1}^{\infty} J^{n-1}\sum_{r:dist(r,i)=n}\   \nu(\|\nabla_{r} f \|^2   )
  \end{align*} 
  for some constant $D_3>0$.   
\end{proof}We can now prove the Spectral Gap type inequality
type inequality for the expectation with respect to the Gibbs measure of the one site variance  $\nu\mathbb{E}^{i}\vert f-\mathbb{E}^{i}f\vert^2$
\begin{lem}\label{teleutSpect1}  Assume   that the  measure $\mu $ satisfies the log-Sobolev inequality and  that the  local specification has quadratic interactions $V$ as in (\ref{quadratic}). Then, for $J$ sufficiently small, the following spectral gap type inequality holds   $$\nu\mathbb{E}^{i}\vert f-\mathbb{E}^{i}f\vert^2\leq D_{4}\nu(\|\nabla_{i} f \|^2   )+D_{4} \sum_{n=1}^{\infty} J^{n-1}\sum_{r:dist(r,i)=n}\   \nu(\|\nabla_{r} f \|^2  $$for some constant $D_4\geq 1$. \end{lem}
 \begin{proof} For any $m\in \mathbb{R}$ we can bound the variance 
  \begin{align}\nonumber \label{4eq12} \mathbb{E}^{i,\omega}\vert f-\mathbb{E}^{i,\omega}f\vert^2\leq &4\mathbb{E}^{i,\omega}\vert f-m\vert^2\\   = &4\mathbb{E}^{i,\omega}\vert f-m\vert^2\mathbb{\mathcal{I}}_{\{\eta(i,\omega) \leq\ L\}}+ 4\mathbb{E}^{i,\omega}\vert f-m\vert^2\mathbb{\mathcal{I}}_{\{\eta(i,\omega) > L\}}\end{align}
where we have again denoted $$\eta(i,\omega)=d( x_{i} )+ \sum_{j\sim i}d(\omega_{j} )$$ Setting $m=\frac{1}{\vert A^{\omega}(L)\vert}\int_{A^{\omega}(L)} f(z)dz$ and   taking  the  expectation with respect to the Gibbs measure in both sides of (\ref{4eq12}) gives
  \begin{align*}\nonumber  \nu\vert f-\mathbb{E}^{i,\omega}f\vert^2\leq4 &\nu\vert f-\frac{1}{\vert A^{\omega}(L)\vert}\int_{A^{\omega}(L)} f(z)dz\vert^2\mathbb{\mathcal{I}}_{\{\eta(i,\omega) \leq\ L\}}+ \\   &4\nu\vert f-\frac{1}{\vert A^{\omega}(L)\vert}\int_{A^{\omega}(L)} f(z)dz\vert^2\mathbb{\mathcal{I}}_{\{\eta(i,\omega) > L\}}\end{align*}
We can bound  the   first and the  second term on the right hand side from  Lemma \ref{PoincareBall} and Lemma \ref{PoincareNOBall} respectively. This leads to 
  \begin{align*}\nonumber  \nu\vert f-\mathbb{E}^{i,\omega}f\vert^2\leq4(D_{L}+D_{3})\nu\| \nabla_{i} f\|^2+ 4D_{3} \sum_{n=1}^{\infty} J^{n-1}\sum_{r:dist(r,i)=n}\   \nu(\|\nabla_{r} f \|^2   )\end{align*}
 which proves the lemma for appropriate positive constant  $D_4$.
 \end{proof}
 If we combine  Lemma \ref{teleutSpect1} and Corollary \ref{4cor2}
we also have \begin{cor}\label{4cor6} Assume    that   $\mu $ satisfies the log-Sobolev inequality and  that the  local specification has quadratic interactions $V$ as in (\ref{quadratic}). Then, for $J$ sufficiently small, the following holds
\begin{align*}  \nu ((f-\mathbb{E}^{i}f)^{2} d^{2}(x_{j}))\leq & D_{5}\nu(\|\nabla_{i} f \|^2   )+D_{5} \sum_{n=1}^{\infty} J^{n-1}\sum_{r:dist(r,i)=n}\   \nu(\|\nabla_{r} f \|^2)     \end{align*} for some constant $D_5\geq 1$.
  \end{cor}
 The following lemma provides   the sweeping out  inequality  for the one site measure
 \begin{lem} \label{4lem7} Assume   that the  measure $\mu $ satisfies the log-Sobolev inequality and  that the  local specification has quadratic interactions $V$ as in (\ref{quadratic}). Then, for $J$ sufficiently small, for every $j \sim i$ \begin{align*}\nonumber  \nu \| \nabla_{j}(\mathbb{E}^{i}f)\|^2\leq  G_{1}  \sum_{n=0}^{\infty} J^{n}\sum_{r:dist(r,j)=n}\   \nu \|\nabla_{r} f\|^2      \end{align*}   for  a constant $G_1\in [1,\infty)$.  \end{lem} 
\begin{proof}Combining Lemma \ref{teleutSpect1} and Lemma  \ref{4lem1} together, for  $J$ sufficiently small, we obtain the following,  
 \begin{align*}
  \nu \| \nabla_{j}(\mathbb{E}^{i}f)\|^2\leq & J^{2}D_{1}D_{4}\nu(\|\nabla_{i} f \|^2   )+J^{2}D_{1}D_{4} \sum_{n=1}^{\infty} J^{n-1}\sum_{r:dist(r,i)=n}\   \nu(\|\nabla_{r} f \|^2 \ +  \\ &   D_{1} \nu \|\nabla_{j}f\|^{2}++J D_{1}      \sum_{n=1}^{\infty} J^{n}\sum_{ dist(r,i) =n}\   \nu(\|\nabla_{r} f\|^2   )\\  \leq & D_{1} \nu\|\nabla_{j}f\|^{2}+  2D_{4} D_{1}J\sum_{n=0}^{\infty} J^{n}\sum_{r:dist(r,i)=n}\   \nu(\|\nabla_{r} f \|^2 \  \end{align*}  because $J<1$ and $D_4\geq 1$.  Since $i\sim j$ that implies that every node $r$ which has distance $n$ from $i$, i.e. $r:dist(r,i)=n$, will have distance $n-1$ or $n+1$ from $j$. So the last inequality becomes. 
 \begin{align*}
  \nu \| \nabla_{j}(\mathbb{E}^{i}f)
\|^2\leq & D_{1} \nu \|\nabla_{j}f\|^{2}+ 4D_{4} D_{1}J\sum_{n=0}^{\infty} J^{n}\sum_{r:dist(r,i)=n}\   \nu(\|\nabla_{r} f \|^2  \end{align*}and the lemma follows for $G_{1}=D_1+4D_1D_4$. 
\end{proof}
Define the following sets
\begin{align*}
&\Gamma_0=(0,0)\cup\{j \in \mathbb{Z}^2 : dist(j,(0,0))=2m \text{\; for some \;}m\in\mathbb{N}\}, \\  
&\Gamma_1=\mathbb{Z}^2\smallsetminus\Gamma_0 .
\end{align*}
where $dist(i,j)$ refers to  the distance of  the shortest path (number of vertices) between two nodes $i$ and $j$.  Note that $dist(i,j)>1$ for all $i,j \in\Gamma_k,k=0,1$ and $\Gamma_0\cap\Gamma_1=\emptyset$.  Moreover $\mathbb{Z}^2=\Gamma_0\cup\Gamma_1$. In the next  proposition we will prove a sweeping out inequality for the product measures $\mathbb{E}^{\Gamma_0}$.    
\begin{prop}  \label{4prop8} Assume   that the  measure $\mu $ satisfies the log-Sobolev inequality and  that the  local specification has quadratic interactions $V$ as in (\ref{quadratic}). Then, for $J$ sufficiently small, the following sweeping out inequality is true
$$\nu \| \nabla_{\Gamma_1}(\mathbb{E}^{\Gamma_0}f)\|^2 \leq R_1\mathcal{\nu}\| \nabla_{\Gamma_1} f\|^2+R_2\nu \| \nabla_{\Gamma_0} f\|^2$$ 
   for constants $R_1\in [1,\infty)$  and $0<R_2\leq \frac{JG_6}{1-4J} <1$. \end{prop}
  \begin{proof}
We can write

 \begin{align}\label{4eq13}\nu \| \nabla_{\Gamma_1}(\mathbb{E}^{\Gamma_{0}}f)\|^2=&\sum_{i\in \Gamma_1} \nu \| \nabla_{i}(\mathbb{E}^{\Gamma_{0}}f)\|^2\leqslant\sum_{i\in \Gamma_1} \nu \| \nabla_{i}(\mathbb{E}^{\{\sim i\}}f)\|^2 \end{align}
\begin{figure}[h]
           \begin{center}
\epsfig{file=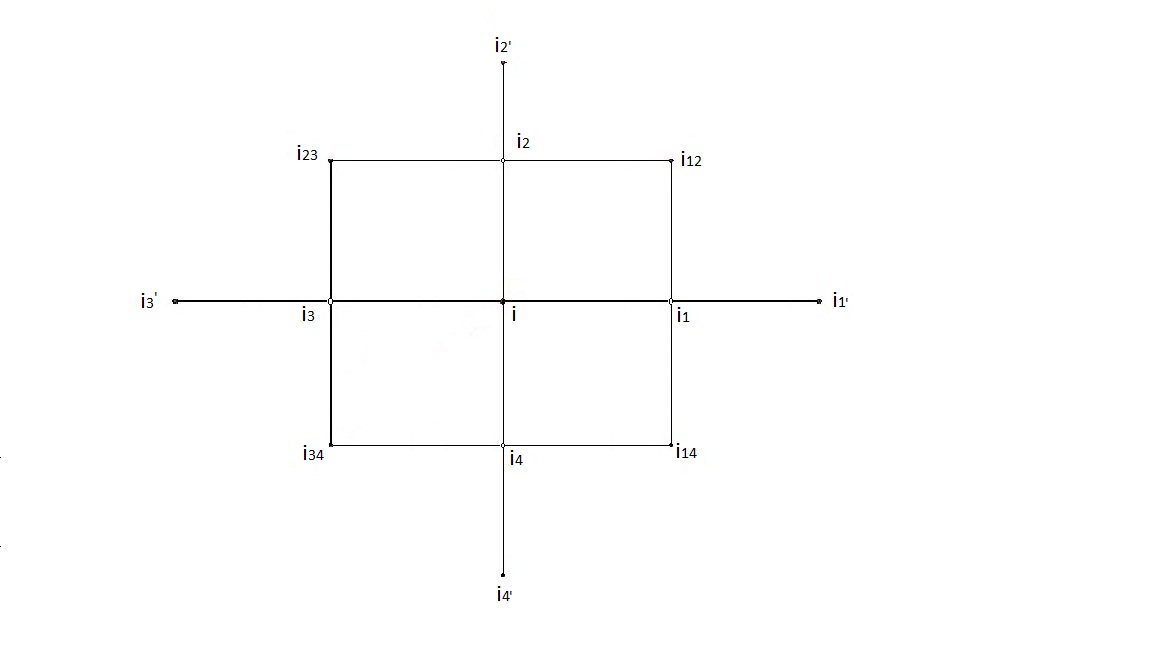, height=4.5cm}
\caption{$\circ = \Gamma_0$, $\bullet = \Gamma_1$}
\label{fig1}
           \end{center}
\end{figure}
If we denote $\{i_{1},i_{2},i_{3},i_{4}\}:=\{\sim i\}$ the neighbours  of  note $i$  as shown on Figure \ref{fig1},  and  use   Lemma \ref{4lem7}
we get the following
 \begin{align}\nonumber\label{4eq14}\nu \| \nabla_{i}(\mathbb{E}^{\{\sim i\}}f)\|^2=&\nu \| \nabla_{i}(\mathbb{E}^{ i_1}\mathbb{E}^{\{ i_{2},i_3,i_4\}}f)\|^2\leq  G_{1} \nu \|\nabla_{i}( \mathbb{E}^{\{ i_{2},i_3,i_4\}}f)\|^{2}+\\  & G_{1}  \sum_{n=1}^{\infty} J^{n}\sum_{r:dist(r,i)=n}\   \nu(\|\nabla_{r} (\mathbb{E}^{\{ i_{2},i_3,i_4\}}f)\|^2)  \end{align}
We will compute  the second term in the right hand side  of (\ref{4eq14}). For $n=1$ , we have
\begin{align}\nonumber\label{4eq15}
 \sum_{dist(r,i)=1}\nu\|\nabla_{r}( \mathbb{E}^{\{ i_{2},i_3,i_4\}}f)\|^{2}= & \nu\|\nabla_{i_{1}}( \mathbb{E}^{\{ i_{2},i_3,i_4\}}f)\|^{2}+\sum_{ r=i_{2},i_3,i_4} \nu \|\nabla_{r}( \mathbb{E}^{\{ i_{2},i_3,i_4\}}f)\|^{2}
\\ \leq  & \nu \|\nabla_{i_1}f\|^{2} \end{align} 
For $n=2$ , we distinguish between the nodes  $r$ in $\{dist(r,i)=2\}$ which neighbour only one of the neighbours $ \{ i_2,i_3,i_4\}$ of $ i$, which are the $ i'_2,i'_3,i'_4,i_{12},i_{14}$, and these which   neighbour two of the  node in $\{ i_2,i_3,i_4\}$, which are the $i_{23}$ and  $i_{34}$ neighbouring $i_2,i_3$ and $i_3,i_4$ respectively, as shown in Figure \ref{fig1}. We can then write
\begin{align}\nonumber\label{4eq16}
 \sum_{dist(r,i)=2}\nu \|\nabla_{r}( \mathbb{E}^{\{ i_{2},i_3,i_4\}}f)\|^{2}&=  \sum_{r= i_{23},i_{34}}\nu \|\nabla_{i}( \mathbb{E}^{\{ i_{2},i_3,i_4\}}f)\|^{2}+\\ &\sum_{ r=i'_2,i'_3,i'_4,i_{12},i_{14}} \nu \|\nabla_{r}( 
\mathbb{E}^{\{ i_{2},i_3,i_4\}}f)\|^{2}
  \end{align} 
To bound the second term on the right hand side of (\ref{4eq16}), for any  $r\in\{i'_2,i'_3,i'_4,i_{12},i_{14} \}$ neighbouring the node $t\in\{ i_{2},i_3,i_4\}$ we use Lemma \ref{4lem7}
 \begin{align*}\nonumber \nu\| \nabla_{r}(\mathbb{E}^{\{ i_{2},i_3,i_4\}}f)\|^2= &\nu\| \nabla_{r}\mathbb{E}^{\{t\}}(\mathbb{E}^{\{ i_{2},i_3,i_4\}\smallsetminus\{t\}}f)\|^2\leq \nu\| \nabla_{r}(\mathbb{E}^{t}f)\|^2 \leq \\ & G_{1}  \sum_{n=0}^{\infty} J^{n}\sum_{s:dist(s,r)=n}\   \nu(\|\nabla_{s} f\|^2)  \end{align*}
which leads to \begin{align*}\nonumber \sum_{ r=i'_2,i'_3,i'_4,i_{12},i_{14}}\nu\| \nabla_{r}(\mathbb{E}^{\{ i_{2},i_3,i_4\}}f)\|^2  \leq G_{1} \sum_{ dist(r,i)=2} \sum_{n=0}^{\infty} J^{n}\sum_{s:dist(s,r)=n}\   \nu(\|\nabla_{s} f\|^2)  \end{align*}
Since for nodes $r:dist(r,i)=2$, the  nodes $s$ such  that $dist(s,r)=n$ have distance from $i$ equal to $\vert n-2\vert$, $n$ or $n+2$ we get 
\begin{align}\nonumber\label{4eq17}\sum_{ r=i'_2,i'_3,i'_4,i_{12},i_{14}}\nu&\| \nabla_{r}(\mathbb{E}^{\{ i_{2},i_3,i_4\}}f)\|^2  \leq  8G_{1}J\sum_{n=0}^{1} \sum_{r:dist(s,i)=n}\   \nu\|\nabla_{s} f\|^2+\\  &   8G_{1} \sum_{n=2}^{\infty}J^{n-2}\sum_{r:dist(r,i)=n}\   \nu\|\nabla_{r} f\|^2  \end{align}
To bound the first term on the right hand side of (\ref{4eq16}), for example for   $r=i_{23}$ neighbouring the nodes $i_{2}$ and $i_3 $ we use again Lemma \ref{4lem7}
 \begin{align} \label{4eq18}\nu\|\nabla_{i_{23}}( \mathbb{E}^{\{ i_{2},i_3,i_4\}}f)\|^{2} \leq \nu\|\nabla_{i_{23}}(\mathbb{E}^{\{ i_{2},i_3\}}f)\|^{2} \leq G_{1}  \sum_{n=0}^{\infty} J^{n}\sum_{s:dist(s,i_{23})=n}\   \nu\|\nabla_{s} (\mathbb{E}^{i_3}f)\|^2  \end{align}  
The first term for $n=0$ on the sum  of (\ref{4eq18})   by Lemma  \ref{4lem7} is bounded by
 \begin{align}\label{4eq19} \nonumber \nu\|\nabla_{i_{23}}(\mathbb{E}^{i_3}f)\|^{2} \leq &G_{1}  \sum_{n=0}^{\infty} J^{n}\sum_{r:dist(s,i_{23})=n}\   \nu\|\nabla_{s} f\|^2 \leq \\ \nonumber \leq &   G_{1}J\sum_{n=0}^{1} \sum_{r:dist(s,i)=n}\   \nu\|\nabla_{s} f\|^2+\\  &   G_{1} \sum_{n=2}^{\infty}J^{n-2}\sum_{r:dist(r,i)=n}\   \nu\|\nabla_{r} f\|^2  \end{align}  
The  terms for $n=1$ on the  sum  of (\ref{4eq18}) become    
\begin{align}\label{4eq20}G_{1}  J\sum_{s:dist(s,i_{23})=1}\   \nu\|\nabla_{s} (\mathbb{E}^{i_3}f)\|^2\leq G_{1}  J\sum_{s:dist(s,i)=1,3}\   \nu\|\nabla_{s} f\|^2
\end{align} The terms for $n=2$ on the  sum  of (\ref{4eq18}) can be divided on those that neighbour $i_{3}$ and those that not 
\begin{align}\label{4eq21}\sum_{s:dist(s,i_{23})=2}\   \nu\|\nabla_{s} (\mathbb{E}^{i_3}f)\|^2=&\sum_{s:dist(s,i_{23})=2,s\sim i_{3}}\   \nu\|\nabla_{s} (\mathbb{E}^{i_3}f)\|^2+\sum_{s:dist(s,i_{23})=2,s\nsim i_{3}}\   \nu\|\nabla_{s} (\mathbb{E}^{i_3}f)\|^2 \end{align}
For the second term on the right hand side of (\ref{4eq21})
\begin{align} \label{4eq22}\sum_{s:dist(s,i_{23})=2,s\nsim i_{3}}\   \nu\|\nabla_{s} (\mathbb{E}^{i_3}f)\|^2\leq\sum_{s:dist(s,i_{23})=2,s\nsim i_{3}}\   \nu\|\nabla_{s} f\|^2  
\end{align}
While
for the first term on the right hand side of (\ref{4eq21})
we can use Lemma \ref{4lem7}
\begin{align}\label{4eq23}\nonumber \sum_{s:dist(s,i_{23})=2,s\sim i_{3}}\   \nu\|\nabla_{s} (\mathbb{E}^{i_3}f)\|^2 \leq  &  G_{1} \sum_{s:dist(s,i_{23})=2,s\sim i_{3}} \sum_{n=0}^{\infty} J^{n}\sum_{r:dist(r,s)=n}\   \nu\|\nabla_{r} f\|^2  \leq \\ \nonumber \leq &   4G_{1}\sum_{n=0}^{1} J^{n}\sum_{r:dist(s,i)=n}\   \nu\|\nabla_{s} f\|^2+\\  &   4G_{1} \sum_{n=2}^{\infty}J^{n-2}\sum_{r:dist(r,i)=n}\   \nu\|\nabla_{r} f\|^2
\end{align}
From (\ref{4eq21})-(\ref{4eq23})
we get the following bound for the terms for $n=2$ on the  sum  of (\ref{4eq18}) \begin{align}\nonumber \label{4eq24}G_{1}J^{2}\sum_{s:dist(s,i_{23})=2}\   \nu\|\nabla_{s} (\mathbb{E}^{i_3}f)\|^2\leq  &   G_{1}J^{2}\sum_{s:dist(s,i)=2 ,4}\   \nu\|\nabla_{s} f\|^2+\\  &   4G^{2}_{1} \sum_{n=0}^{\infty}J^{n}\sum_{r:dist(r,i)=n}\   \nu\|\nabla_{r} f\|^2 \end{align}
Finally,  for  the terms for $n>2$ on the sum on the right hand side of  (\ref{4eq18}), we get\begin{align}   \label{4eq25} G_{1}\sum_{n=3}^{\infty} J^{n}\sum_{r:dist(s,i_{23})=n}\   \nu\|\nabla_{s} (\mathbb{E}^{i_3}f)\|^2\leq G_{1}\sum_{n=3}^{\infty} J^{n-2}\sum_{r:dist(s,i)=n}\   \nu\|\nabla_{s} f\|^2
\end{align}
   If we put (\ref{4eq19}), (\ref{4eq20}), (\ref{4eq24}) and (\ref{4eq25}) in (\ref{4eq18})
we get \begin{align} \label{4eq26}\nonumber  \nu\|\nabla_{i_{23}}( \mathbb{E}^{\{ i_{2},i_3,i_4\}}f)\|^{2} \leq &   G_2\sum_{n=0}^{1} J^{n}\sum_{r:dist(s,i)=n}\   \nu\|\nabla_{s} f\|^2+\\   & G_{2}\sum_{n=2}^{\infty} J^{n-2}\sum_{r:dist(s,i)=n}\   \nu\|\nabla_{s} f\|^2  \end{align}
  for $G_2=4G_1^2+3G_1$. Exactly the same bound can be obtain for the other term, $\nu\|\nabla_{i_{34}}( \mathbb{E}^{\{ i_{2},i_3,i_4\}}f)\|^{2} $, on the  first  sum of the right hand side of (\ref{4eq16}).
Gathering together, (\ref{4eq16}), (\ref{4eq17}) and (\ref{4eq26})
\begin{align}\nonumber\label{4eq27}
 \sum_{dist(r,i)=2}\nu \|\nabla_{r}( \mathbb{E}^{\{ i_{2},i_3,i_4\}}f)\|^{2}=   &   G_{3}\sum_{n=0}^{1} J^{n}\sum_{r:dist(s,i)=n}\   \nu\|\nabla_{s} f\|^2+\\   & G_{3}\sum_{n=2}^{\infty} J^{n-2}\sum_{r:dist(s,i)=n}\   \nu\|\nabla_{s} f\|^2  
  \end{align} 
for $G_3=8G_1+2G_2$. 

Furthermore, for every $r:dist(r,i)\geq 3$,
\begin{align}\label{4eq28}\nu\|\nabla_{r} (\mathbb{E}^{\{ i_{2},i_3,i_4\}}f)\|^2\leq\nu\|\nabla_{r} f\|^2\end{align} Finally, if we put (\ref{4eq15}) and (\ref{4eq27}) and (\ref{4eq28}) in (\ref{4eq14}) we obtain 
 \begin{align}\label{4eq29}\nonumber\nu\| \nabla_{i}(\mathbb{E}^{\{\sim i\}}f)\|^2=&\nu\|\nabla_{i}( \mathbb{E}^{\{i_{1} ,i_{2},i_3,i_4\}}f)\|^{2}\leq G_{1} \nu\|\nabla_{i}( \mathbb{E}^{\{ i_{2},i_3,i_4\}}f)\|^{2}+\\ &   G_{4} \sum_{n=0}^{\infty}J^{n}\sum_{r:dist(r,i)=n}\   \nu\|\nabla_{r} f\|^2  \end{align} 
 for constant $G_4=G_1G_3+G_1$.

If we repeat the  same calculation  recursively for the first term on the right hand side of (\ref{4eq29}), then for   $\nu\|\nabla_{i}( \mathbb{E}^{\{i_3,i_4\}}f)\|^{2}$ and $\nu\|\nabla_{i}( \mathbb{E}^{\{ i_4\}}f)\|^{2}$ we will finally obtain  
 \begin{align*}\nonumber\nu\| \nabla_{i}(\mathbb{E}^{\{\sim i\}}f)\|^2\leq G_5 \sum_{n=0}^{\infty} J^{n}\sum_{r:dist(s,i)=n}\   \nu\|\nabla_{s} f\|^2  \end{align*}
for a constant $G_5>1$. From the last inequality and (\ref{4eq13})
we get 
 \begin{align*}\nu\| \nabla_{\Gamma_1}(\mathbb{E}^{\Gamma_0}f)
\|^2 \leq R_1\mathcal{\nu}\| \nabla_{\Gamma_1} f
\|^2+R_2\nu\| \nabla_{\Gamma_0} f
\|^2 \end{align*}
for a constant $R_2\leq G_5J(\sum_{k=0}^{\infty} (4J)^k)\leq\frac{JG_5}{1-4J}<1 $ for $J$ sufficiently small such that $J<\min\{\frac{1}{4},\frac{1}{C_5+4}\}$.
  \end{proof}
  
 \section{\label{secLS}Second Sweeping out relations.} 
In this section we prove the second sweeping out relation. We start by first proving in the next lemma the second sweeping out relation between two neighbouring nodes.\begin{lem} \label{5lem1}Assume   that the  measure $\mu $ satisfies the log-Sobolev inequality and  that the  local specification has quadratic interactions $V$ as in (\ref{quadratic}). Then, for $J$ sufficiently small, for every $i\sim j$ the following  sweeping out inequality holds   \begin{align*}\nu\| \nabla_{i}(\mathbb{E}^{j}\vert f\vert^2)^\frac{1}{2}
\|^2 \leq    R_{3}    \sum_{n=0}^{\infty} J^{n}\sum_{ dist(r,i) =n}\   \nu\|\nabla_{r} f\|^2 \end{align*}for $R_{3}\geq 1 $. 
 \end{lem}
 \begin{proof}  Consider the (sub)gradient $\nabla_i=(X_1^i,X_2^i,...,X_N^i)$. We can then write
\begin{align}\label{5eq1} \| \nabla_i(\mathbb{E}^j f^2)^{\frac{1}{2}}\|^2 =\sum_{k=1}^N (X_k^i(\mathbb{E}^j f^2)^{\frac{1}{2}})^2
\end{align}
Then for every $k\in\{1,...,N\}$ we can compute  
\begin{align}\label{5eq2}\vert  X^k_i(\mathbb{E}^{j}f^2)^\frac{1}{2}
\vert^2=  &\vert \frac{1}{2}(\mathbb{E}^{j}f^2)^{\frac{1}{2}-1}
X^k_i(\mathbb{E}^{j}f^2)\vert^2= \frac{1}{4}  (\mathbb{E}^{j}f^2)^{-1}
\vert X^k_i(\mathbb{E}^{j}f^2)\vert^2\end{align} 
But from relationship~\eqref{4eq2} of  Lemma~\ref{4lem1}, if we put $f^2$ in $f$, we have 
 \begin{align}\vert X^k_i(\mathbb{E}^{j}f^2)\vert^{2}\leq   \label{5eq3}2\vert\int  X^k_i(f^2)\rho_{j} dx_{j}\vert^{2}+ 2\vert\int  f^2(X^k_i\rho_{j}) dx_{j}\vert^{2}\end{align}
 where again   $\rho_{j}$ denotes  the density of
$\mathbb{E}^{j}$. For the second term in~\eqref{5eq3} we  have
\begin{equation}\label{5eq4} \vert\int  f^2(X^k_i\rho_{j}  ) dx_{j}\vert^{2}\leq J^{2}\vert\mathbb{E}^{j}(f^2; X^k_iV(x_{j},x_{i}))\vert^2\end{equation}
While for the   first term of~\eqref{5eq3} the  following   bound holds
\begin{align} \vert\int X^k_i(f^2)\rho_{j} dx_{j}\vert^{2}&=2\vert\mathbb{E}^{j}(f(X^k_if)) \vert^{2}\leq
2\left(  \mathbb{E}^{j}f^{2} \right)\left(\mathbb{E}^{j}\vert X^k_if\vert^2\right)
\label{5eq5}\end{align}
 where above we used the   Cauchy-Swartz \"   inequality.
If we plug~\eqref{5eq4} and~\eqref{5eq5} in~\eqref{5eq3} we get
\begin{align}\label{5eq6}\vert X^k_i(\mathbb{E}^{j}f^2)\vert^2&\leq 4\left(  \mathbb{E}^{j}f^{2} \right)\left(\mathbb{E}^{j}\vert X^k_i f\vert^2\right)+2J^{2}\vert\mathbb{E}^{j}(f^2; X^k_iV(x_{j},x_{i})))\vert
^2\end{align}
  Combining together (\ref{5eq2}) and (\ref{5eq6}) we obtain  
  \begin{align}\label{5eq7}\vert X^k_i(\mathbb{E}^{j}f^2)^\frac{1}{2}
\vert^2\leq  \mathbb{E}^{j}\vert X^k_i f\vert^2+J^{2}(\mathbb{E}^{j}f^2)^{-1}\vert\mathbb{E}^{j}(f^2; X^k_iV(x_{j},x_{i})))\vert
^2 \end{align}
In order to calculate the second term on the right hand side of (\ref{5eq7}) we will use the following lemma.\begin{lem}\label{5lem2}  For any probability measure $\mu$ the following inequality holds  

\begin{align*}
\mu(\vert f\vert^2;  g)
 \leq   \tilde c\left(\mu\vert f\vert^2\right)^{\frac{1}{2}}\left( \mu(\vert f-\mu f\vert^2(\vert g\vert^{2}+\mu \vert g\vert^{2})) \right)^\frac{1}{2}\end{align*} for some constant $\tilde c$ uniformly on  the  boundary conditions. \end{lem} 
Without loose of generality we can assume $\tilde c \geq 1$. The proof of  Lemma \ref{5lem2} can be found in \cite{Pa1}.  Applying this bound to the second term in (\ref{5eq7}) leads to                  \begin{align}\nonumber \left(\mathbb{E}^{j}f^2\right)^{-1}\vert\mathbb{E}^{j}(f^2;X^k_iV(x_{j}, x_{i}))\vert^2\leq \tilde c^{2} \mathbb{E}^{j}\left[\vert f-\mathbb{E}^{j}f\vert^2 \left((X^k_iV(x_{j},x_{i}))^2+\mathbb{E}^{j}(X^k_iV(x_{j},x_{i}))^2\right)\right] \end{align}
From the last inequality and (\ref{5eq7})
we have  \begin{align*}\vert X^k_i(\mathbb{E}^{j}f^2)^\frac{1}{2}
\vert^2\leq  &\mathbb{E}^{j}\vert X^k_if\vert^2+J^{2}\tilde c^{2} \mathbb{E}^{j}\left[\vert f-\mathbb{E}^{j}f\vert^2 (X^k_iV(x_{j},x_{i}))^2 \right]+\\  &J^{2}\tilde c^{2} \mathbb{E}^{j}\left[\vert f-\mathbb{E}^{j}f\vert^2 \mathbb{E}^{j}(X^k_iV(x_{j},x_{i}))^2\right] \end{align*}
Putting    this in (\ref{5eq1}) leads to
 \begin{align*}  \| \nabla_i(\mathbb{E}^j f^{2})^\frac{1}{2}
\|^2=  &\mathbb{E}^{j}\| \nabla_i f\|^2+J^{2}\tilde c^{2} \mathbb{E}^{j}\left[\vert f-\mathbb{E}^{j}f\vert^2  \|\nabla_iV(x_{j},x_{i})\|^2 \right]+\\  &J^{2}\tilde c^{2} \mathbb{E}^{j}\left[\vert f-\mathbb{E}^{j}f\vert^2 \mathbb{E}^{j} \|\nabla_iV(x_{j},x_{i})\|^2\right]\end{align*}
 If we take the expectation with respect to the Gibbs measure  and   bound $\|\nabla_iV(x_{j},x_{i})\|^2$ by (\ref{quadratic})   we get 
 \begin{align*} \nu \| \nabla_i(\mathbb{E}^j f^{2})^\frac{1}{2}
\|^2\leq&\nu \| \nabla_i f\|^2+2kJ^{2}\tilde c^{2} \nu \left[\vert f-\mathbb{E}^{j}f\vert^2  d^2(x_i)  \right]+\\ \nonumber &kJ^{2}\tilde c^{2} \nu \left[\vert f-\mathbb{E}^{j}f\vert^2  d^2(x_j) \right]+kJ^{2}\tilde c^{2} \nu \left[\vert f-\mathbb{E}^{j}f\vert^2  \mathbb{E}^{j}d^2(x_j)\right]\end{align*}
 If we bound the second  and third term on the right hand side  by Corollary \ref{4cor6} we get\begin{align}\label{5eq8}\nonumber \nu \| \nabla_i(\mathbb{E}^j f^{2})^\frac{1}{2}
\|^2\leq & \nu \| \nabla_i f\|^2+3kJ^{2}\tilde c^{2}   D_{5}\nu(\|\nabla_{j} f \|^2   )+\\ & \nonumber 3kJ^{2}\tilde c^{2} D_{5} \sum_{n=1}^{\infty} J^{n-1}\sum_{r:dist(r,j)=n}\   \nu(\|\nabla_{r} f \|^2)+\\  & kJ^{2}\tilde c^{2} \nu \left[\vert f-\mathbb{E}^{j}f\vert^2  \mathbb{E}^{j}d^2(x_j)\right]\end{align}
For the last term  on the right hand side of (\ref{5eq8}) we can write
\begin{align*}\nu \left[\vert f-\mathbb{E}^{j}f\vert^2  \mathbb{E}^{j}d^2(x_j)\right]=\nu \left[\mathbb{E}^{j}(\vert f-\mathbb{E}^{j}f\vert^2  )d^2(x_j)\right]
\end{align*}
and now apply  the U-bound inequality (\ref{Ubound}) of Proposition \ref{proposition}
\begin{align*} \nu \left[\vert f-\mathbb{E}^{j}f\vert^2  \mathbb{E}^{j}d^2(x_j)\right]\leq & C  \nu(\mathbb{E}^{j}(\vert f-\mathbb{E}^{j}f\vert^2  ))+\nonumber \\ & C \sum_{n=0}^{\infty} J^{n}\sum_{r:dist(r,j)=n}\   \nu\|\nabla_{r} (\mathbb{E}^{j}\vert f-\mathbb{E}^{j}f\vert^2  )^{\frac{1}{2}} \|^2      \end{align*} In order to bound the variance on the first term on the right  hand side we can use the spectral gap type inequality of   Lemma \ref{teleutSpect1}  
\begin{align}\label{5eq9} \nu \left[\vert f-\mathbb{E}^{j}f\vert^2  \mathbb{E}^{j}d^2(x_j)\right]\leq & CD_{4}\nu\|\nabla_{j} f \|^2   +C  D_4\sum_{n=1}^{\infty} J^{n-1}\sum_{r:dist(r,j)=n}\   \nu \|\nabla_{r} f^2\|+\nonumber \\ & C \sum_{n=0}^{\infty} J^{n}\sum_{r:dist(r,j)=n}\   \nu\|\nabla_{r} (\mathbb{E}^{j}\vert f-\mathbb{E}^{j}f\vert^2  )^{\frac{1}{2}} \|^2   \end{align} For $n=0$ the term of the second  sum is zero, while for $n>1$  the nodes  do not neighbour with $j$, so we have   
\begin{align*}   \|\nabla_{r} (\mathbb{E}^{j}(\vert f-\mathbb{E}^{j}f\vert^2  ))^{\frac{1}{2}}\|^2 =&\sum_{k=1}^N \vert X_{k}^r (\mathbb{E}^{j}(\vert f-\mathbb{E}^{j}f\vert^2  ))^{\frac{1}{2}} \vert^2=  \\ &  \frac{1}{4}\sum_{k=1}^{N} \vert (\mathbb{E}^{j}(\vert f-\mathbb{E}^{j}f\vert^2  ))^{-\frac{1}{2}} X^{r}_{k}(\mathbb{E}^{j}(\vert f-\mathbb{E}^{j}f\vert^2  ) \vert^2 \leq \\ &\sum_{k=1}^{N} \vert (\mathbb{E}^{j}(\vert f-\mathbb{E}^{j}f\vert^2  ))^{-\frac{1}{2}}\mathbb{E}^{j}\left[( f-\mathbb{E}^{j}f)\left(  \vert X^{r}_{k} f \vert+\mathbb{E}^{j}  \vert X^{r}_{k} f \vert\right) \right] \vert^2  
\end{align*}
From   Cauchy-Swartz inequality  the last becomes
 \begin{align}\label{5eq10}   \|\nabla_{r} (\mathbb{E}^{j}(\vert f-\mathbb{E}^{j}f\vert^2  ))^{\frac{1}{2}}\|^2 \leq&2\sum_{k=1}^{N}\mathbb{E}^{j}  \vert X^{r}_{k} f \vert^{2}=2\mathbb{E}^{j} \| \nabla_r f\|^{2} 
\end{align}
Putting  together (\ref{5eq9}) and (\ref{5eq10})
\begin{align}\label{5eq11}\nonumber \nu \left[\vert f-\mathbb{E}^{j}f\vert^2  \mathbb{E}^{j}d^2(x_j)\right]\leq &C   J\sum_{ r\sim j}\   \nu\|\nabla_{r} (\mathbb{E}^{j}\vert f-\mathbb{E}^{j}f\vert^2  )^{\frac{1}{2}}\|^2+CD_{4}\nu\|\nabla_{j} f\|^2+\\  &(C  D_4+2C)      \sum_{n=1}^{\infty} J^{n-1}\sum_{ dist(r,j) =n}\   \nu\|\nabla_{r} f\|^2   \end{align} 
From (\ref{5eq8}) and  (\ref{5eq11}) we  get
\begin{align}\nonumber \label{5eq12}\nu \| \nabla_i(\mathbb{E}^j f^{2})^\frac{1}{2}
\|^2\leq & D_{6}\nu \| \nabla_i f\|^2+D_{6}J^{3}+D_{6}J      \sum_{n=0}^{\infty} J^{n}\sum_{ dist(r,j) =n}\   \nu\|\nabla_{r} f\|^2+\\ &\sum_{ r\sim j}\   \nu\|\nabla_{r} (\mathbb{E}^{j}\vert f-\mathbb{E}^{j}f\vert^2  )^{\frac{1}{2}}\|^2\end{align}
for a  constant $D_6=1+  k\tilde c^{2}(3D_{5}+C  D_4+2C)$.  If we replace $f$ by $ f-\mathbb{E}^{j}f$ in (\ref{5eq12}) we get
\begin{align*}\nonumber \nu \| \nabla_i(\mathbb{E}^j (f-\mathbb{E}^{j}f)^{2})^\frac{1}{2}
\|^2\leq & D_{6}\nu \| \nabla_i (f-\mathbb{E}^{j}f)\|^2+D_{6}J^{3}\sum_{ r\sim j}\   \nu\|\nabla_{r} (\mathbb{E}^{j}\vert f-\mathbb{E}^{j}f\vert^2  )^{\frac{1}{2}}\|^2+\\ \nonumber &D_{6}J      \sum_{n=0}^{\infty} J^{n}\sum_{ dist(r,j) =n}\   \nu\|\nabla_{r} (f-\mathbb{E}^{j}f)\|^2\leq \\ & 2D_{6}\nu \| \nabla_i f\|^2+ 2D_{6}\nu \| \nabla_i (\mathbb{E}^{j}f)\|^2+D_{6}J^{3}\sum_{ r\sim j}\   \nu\|\nabla_{r} (\mathbb{E}^{j}\vert f-\mathbb{E}^{j}f\vert^2  )^{\frac{1}{2}}\|^2+\\  & 2D_{6}J^2\sum_{ dist(r,j) =1}\   \nu\|\nabla_{r} (\mathbb{E}^{j}f)\|^2+ 2D_{6}J      \sum_{n=0}^{\infty} J^{n}\sum_{ dist(r,j) =n}\   \nu\|\nabla_{r} f\|^2 \end{align*}
If we use Lemma \ref{4lem7}
to bound the second and fourth term in the right hand side of the last inequality we obtain\begin{align*}\nonumber  \nu \| \nabla_i(\mathbb{E}^j (f-\mathbb{E}^{j}f)^{2})^\frac{1}{2}
\|^2\leq   & 2D_{6}(G_{1} +1) \sum_{n=0}^{\infty} J^{n}\sum_{r:dist(r,i)=n}\   \nu \|\nabla_{r} f\|^2+\\ & D_{6}J^{3}\sum_{ r\sim j}\   \nu\|\nabla_{r} (\mathbb{E}^{j}\vert f-\mathbb{E}^{j}f\vert^2  )^{\frac{1}{2}}\|^2+\\  & 2G_{1}D_{6}J^2\sum_{ dist(r,j) =1}\     \sum_{n=0}^{\infty} J^{n}\sum_{r:dist(s,r)=n}\   \nu \|\nabla_{s} f\|^2+ \\ &2D_{6}J       \sum_{n=0}^{\infty} J^{n}\sum_{ dist(r,j) =n}\   \nu\|\nabla_{r} f\|^2 \end{align*}
Since $i$ and $j$ are neighbours and the $r$'s in the sum of the third term have distance less or equal to two from $i$, we can write \begin{align*}\nonumber  \nu \| \nabla_i(\mathbb{E}^j (f-\mathbb{E}^{j}f)^{2})^\frac{1}{2}
\|^2\leq   & D_7  \sum_{n=0}^{\infty} J^{n}\sum_{r:dist(r,i)=n}\   \nu \|\nabla_{r} f\|^2+\\ & D_{7}J^{3}\sum_{ r\sim j}\   \nu\|\nabla_{r} (\mathbb{E}^{j}\vert f-\mathbb{E}^{j}f\vert^2  )^{\frac{1}{2}}\|^2 \end{align*}
for a constant $D_7=24D_{6}G_{1}+2D_6$. If we take the sum over all $i$ such that $i\sim j$ we get 
 \begin{align*}\nonumber \sum_{ r\sim j}\nu \| \nabla_r(\mathbb{E}^j \vert f-\mathbb{E}^{j}f\vert^2)^\frac{1}{2}
\|^2\leq  & \nonumber   4 D_{7}         \sum_{n=1}^{\infty} J^{n-1}\sum_{r:dist(r,j)=n}\   \nu\|\nabla_{r} f\|^2 \\  & 4D_7J\nu\|\nabla_{j} f\|^2+ \\  & 4D_{7}J^{3}\sum_{ r\sim j}\   \nu\|\nabla_{r} (\mathbb{E}^{j}\vert f-\mathbb{E}^{j}f\vert^2  )^{\frac{1}{2}}\|^2\end{align*}
For $J$ sufficiently small so that $\frac{1}{1-4D_{6}J^{3}}<2$ we obtain
 \begin{align*}\label{quad4.22}\nonumber \sum_{ r\sim j}\nu &\| \nabla_i(\mathbb{E}^j \vert f-\mathbb{E}^{j}f\vert^2)^\frac{1}{2}
\|^2\leq   8 D_{7}         \sum_{n=1}^{\infty} J^{n-1}\sum_{r:dist(s,j)=n}\   \nu\|\nabla_{s} f\|^2+8D_{7}J\nu\|\nabla_{j} f\|^2 \end{align*}
If we use the last inequality to bound the last term on the right hand side of (\ref{5eq12}) we obtain \begin{align*}\nonumber \nu \| \nabla_i(\mathbb{E}^j f)^\frac{1}{2}
\|^2\leq & D_{6}\nu \| \nabla_i f\|^2+8D_{7}D_{6}J^{3}         \sum_{n=1}^{\infty} J^{n-1}\sum_{r:dist(s,i)=n}\   \nu\|\nabla_{s} f\|^2+\\  &D_{6}J      \sum_{n=0}^{\infty} J^{n}\sum_{ dist(r,j) =n}\   \nu\|\nabla_{r} f\|^2+8D_{6}D_{7}J^{4}\nu\|\nabla_{j} f\|^2\\ \leq &    2D_{6}\nu \| \nabla_i f\|^2+ (16D_{7}D_{6}J^{2} +D_6)        \sum_{n=1}^{\infty} J^{n}\sum_{r:dist(s,i)=n}\   \nu\|\nabla_{s} f\|^2 \end{align*}
where above we used (\ref{newJnew}). This  finishes the proof for an appropriate constant $R_3$.
 \end{proof}    
 In the next proposition we will extend the sweeping out relations of the last lemma from the two neighboring nodes to the two infinite dimensional disjoint sets $\Gamma_0$ and $\Gamma_1$.
 \begin{prop}\label{5lem3}Assume   that the  measure $\mu $ satisfies the log-Sobolev inequality and  that the  local specification has quadratic interactions $V$ as in (\ref{quadratic}). Then, for $J$ sufficiently small,  the following sweeping out inequality holds
\begin{equation} \label{5eq13}\nu \| \nabla_{\Gamma_i}(\mathbb{E}^{\Gamma_{j}}f^2)^\frac{1}{2}
\|^2\leq C_1\nu\| \nabla_{\Gamma_i}f\|^2+C_2\nu\| \nabla_{\Gamma_j}f\|^2\end{equation}
for $\{i,j\}=\{0,1\}$ and  constants $C_1\in [1,\infty)$ and $0<C_2<1$.\end{prop}
\begin{proof}
The proof will follow the same  lines of the proof of Proposition \ref{4prop8}. If we denote $\{\sim i\}=\{i_{1},i_{2},i_{3},i_{4}\}$ the neighbours   of  note $i$, then we can write

 \begin{align}\label{5eq14}\nu\| \nabla_{\Gamma_1}(\mathbb{E}^{\Gamma_{0}}f^{2})
^\frac{1}{2}\|^2=&\sum_{i\in \Gamma_1} \nu\| \nabla_{i}(\mathbb{E}^{\Gamma_{0}}f^{2})
^\frac{1}{2}\|^2\leqslant\sum_{i\in \Gamma_1} \nu\| \nabla_{i}(\mathbb{E}^{\{\sim i\}}f^{2})^\frac{1}{2}
\|^2 \end{align}
We can   use   Lemma \ref{5lem1}
to bound the last one \begin{align}\nonumber\label{5eq15}\nu\| \nabla_{i}(\mathbb{E}^{\{\sim i\}}f^{2})^\frac{1}{2}
\|^2=&\nu\| \nabla_{i}(\mathbb{E}^{ i_1}\mathbb{E}^{\{ i_{2},i_3,i_4\}}f^{2})^\frac{1}{2}
\|^2\leq R_{3}\nu\| \nabla_i (\mathbb{E}^{\{ i_{2},i_3,i_4\}}f^{2})^\frac{1}{2}
\|^2+ \\ &  R_{3}    \sum_{n=1}^{\infty} J^{n}\sum_{ dist(r,i) =n}\nu\|  \nabla_{r}(\mathbb{E}^{\{ i_{2},i_3,i_4\}}f^{2})^\frac{1}{2}
\|^{2}\end{align}
We will compute  the second term in the right hand side  of (\ref{5eq15}). For $n=1$ , we have
\begin{align}\nonumber\label{5eq16}
 \sum_{dist(r,i)=1}\nu\|\nabla_{r}( \mathbb{E}^{\{ i_{2},i_3,i_4\}}f^{2})^{\frac{1}{2}}\|^{2}= & \nu\|\nabla_{i_{1}}( \mathbb{E}^{\{ i_{2},i_3,i_4\}}f^{2})^{\frac{1}{2}}\|^{2}+\sum_{ r=i_{2},i_3,i_4} \nu\|\nabla_{r}( \mathbb{E}^{\{ i_{2},i_3,i_4\}}f^{2})^{\frac{1}{2}}\|^{2}
\\ \leq  & \nu\|\nabla_{i_1}f\|^{2} \end{align} 
For $n=2$ , we distinguish between the nodes  $r$ in $\{dist(r,i)=2\}$ which neighbour only one of the neighbours $ \{ i_2,i_3,i_4\}$ of $ i$, which are the $ i'_2,i'_3,i'_4,i_{12},i_{14}$, and these which   neighbour two of the  node in $\{ i_2,i_3,i_4\}$, which are the $i_{23k}$ and  $i_{34}$ neighbouring $i_2,i_3$ and $i_3,i_4$ respectively, as shown in Figure \ref{fig1}. We can then write
\begin{align}\nonumber\label{5eq17}
 \sum_{dist(r,i)=2}\nu\|\nabla_{r}( \mathbb{E}^{\{ i_{2},i_3,i_4\}}f^{2})^{\frac{1}{2}}\|^{2}&=  \sum_{r= i_{23},i_{34}}\nu\|\nabla_{i}( \mathbb{E}^{\{ i_{2},i_3,i_4\}}f^{2})^{\frac{1}{2}}\|^{2}+\\ &\sum_{ r=i'_2,i'_3,i'_4,i_{12},i_{14}} \nu\|\nabla_{r}( 
\mathbb{E}^{\{ i_{2},i_3,i_4\}}f^{2})^{\frac{1}{2}}\|^{2}
  \end{align} 
To bound the second term on the right hand side of (\ref{5eq17}), for any  $r\in\{i'_2,i'_3,i'_4,i_{12},i_{14} \}$ neighbouring a  node $t\in\{ i_{2},i_3,i_4\}$ we use Lemma \ref{5lem1}
 \begin{align*}\nonumber \nu\| \nabla_{r}(\mathbb{E}^{\{ i_{2},i_3,i_4\}}f^{2})^{\frac{1}{2}}\|^2= &\nu\| \nabla_{r}\mathbb{E}^{\{t\}}(\mathbb{E}^{\{ i_{2},i_3,i_4\}\smallsetminus\{t\}}f^{2})^{\frac{1}{2}}\|^2\leq \nu\| \nabla_{r}(\mathbb{E}^{t}f^{2})^{\frac{1}{2}}\|^2 \leq \\ &  R_{3}  \sum_{n=0}^{\infty} J^{n}\sum_{r:dist(s,r)=n}\   \nu(\|\nabla_{s} f\|^2)  \end{align*}
which leads to \begin{align*}\nonumber \sum_{ r=i'_2,i'_3,i'_4,i_{12},i_{14}}\nu\| \nabla_{r}(\mathbb{E}^{\{ i_{2},i_3,i_4\}}f^{2})^{\frac{1}{2}}\|^2  \leq\ & R_{3} \sum_{ dist(r,i)=2} \sum_{n=0}^{\infty} J^{n}\sum_{r:dist(s,r)=n}\   \nu\|\nabla_{s} f\|^2  \end{align*}
Since for nodes $r:dist(r,i)=2$, the nodes $s$ such that $dist(s,r)=n$ have distance from $i$ equal to $n-2,n$ or $n+2$ we get \begin{align}\nonumber\label{5eq18}\sum_{ r=i'_2,i'_3,i'_4,i_{12},i_{14}}\nu&\| \nabla_{r}(\mathbb{E}^{\{ i_{2},i_3,i_4\}}f^{2})^{\frac{1}{2}}\|^2  \leq 8R_{3}J^{2}  \nu\|\nabla_{i} f\|^2+8R_{3}J\sum_{r:dist(s,i)=1}\   \nu\|\nabla_{s} f\|^2+\\  &  8R_{3} \sum_{n=2}^{\infty} J^{n-2}\sum_{r:dist(s,i)=n}\   \nu\|\nabla_{s} f\|^2  \end{align}
To bound   the first term on the right hand side of (\ref{5eq17}), for example for   $r=i_{23}$ neighbouring the nodes $i_{2}$ and $i_3 $ we   use again Lemma \ref{5lem1}
 \begin{align}\nonumber \label{5eq19}\nu\|\nabla_{i_{23}}( \mathbb{E}^{\{ i_{2},i_3,i_4\}}f^{2})^{\frac{1}{2}}\|^{2} \leq & \nu\|\nabla_{i_{23}}(\mathbb{E}^{\{ i_{2},i_3\}}f^{2})^{\frac{1}{2}}\|^{2} \leq \\ &  R_{3} \nu\|\nabla_{i_{23}}(\mathbb{E}^{i_3}f^{2})^{\frac{1}{2}}\|^{2}+ R_{3} \sum_{n=1}^{\infty} J^{n}\sum_{r:dist(s,i_{23})=n}\   \nu\|\nabla_{s} (\mathbb{E}^{i_3}f^{2})^{\frac{1}{2}}\|^2  \end{align}  
The first term on the right hand side of (\ref{5eq19})   by Lemma  \ref{5lem1} is bounded by
 \begin{align}\label{5eq20} \nonumber \nu\|\nabla_{i_{23}}(\mathbb{E}^{i_3}f^{2})^{\frac{1}{2}}\|^{2} \leq & R_{3} \sum_{n=0}^{\infty} J^{n}\sum_{r:dist(s,i_{23})=n}\   \nu\|\nabla_{s} f\|^2 \leq \\ \nonumber \leq &  R_{3}J^{2}  \nu\|\nabla_{i} f\|^2+R_{3}J\sum_{r:dist(s,i)=1}\   \nu\|\nabla_{s} f\|^2+  \\ & R_{3}  \sum_{n=2}^{\infty} J^{n-2}\sum_{r:dist(s,i)=n}\   \nu(\|\nabla_{s} f\|^2) \end{align}  
The term for $n=1$ in the sum in the  second term on the right hand side of  (\ref{5eq19}) becomes  
\begin{align}  \nonumber \label{5eq21}R_3 J\sum_{r:dist(s,i_{23})=1}\   \nu\|\nabla_{s} (\mathbb{E}^{i_3}f^{2})^{\frac{1}{2}}\|^2\leq & R_3J\sum_{r:dist(s,i_{23})=1}\   \nu\|\nabla_{s} f\|^2\leq\\  \leq & R_3J\sum_{r:dist(s,i)=1}\   \nu\|\nabla_{s} f\|^2+  R_3J\sum_{r:dist(s,i)=3}\   \nu\|\nabla_{s} f\|^2 
\end{align}
The terms for $n=2$ on the  sum  of (\ref{5eq19}) can be divided on those that neighbour $i_{23}$ and those that not 
\begin{align}\label{5eq22}\nonumber\sum_{s:dist(s,i_{23})=2}\   \nu\|\nabla_{s} (\mathbb{E}^{i_3}f^{2})^{\frac{1}{2}}\|^2=&\sum_{s:dist(s,i_{23})=2,s\sim i_{3}}\   \nu\|\nabla_{s} (\mathbb{E}^{i_3}f^{2})^{\frac{1}{2}}\|^2+\\ &\sum_{s:dist(s,i_{23})=2,s\nsim i_{3}}\   \nu\|\nabla_{s} (\mathbb{E}^{i_3}f^{2})^{\frac{1}{2}}\|^2 \end{align}
For the second term on the right hand side of (\ref{5eq22})
\begin{align} \label{5eq23}\sum_{s:dist(s,i_{23})=2,s\nsim i_{3}}\   \nu\|\nabla_{s} (\mathbb{E}^{i_3}f^{2})^{\frac{1}{2}}\|^2\leq\sum_{s:dist(s,i_{23})=2,s\nsim i_{3}}\   \nu\|\nabla_{s} f\|^2  
\end{align}
while
for the first term on the right hand side of (\ref{5eq22})
we can use Lemma  \ref{5lem1}
\begin{align}\label{5eq24}\nonumber \sum_{s:dist(s,i_{23})=2,s\sim i_{3}}\   \nu\|\nabla_{s} (\mathbb{E}^{i_3}f^{2})^{\frac{1}{2}}\|^2 \leq  &  R_{3} \sum_{s:dist(s,i_{23})=2,s\sim i_{3}} \sum_{n=0}^{\infty} J^{n}\sum_{r:dist(r,s)=n}\   \nu\|\nabla_{r} f\|^2  \leq \\ \nonumber \leq &   4 R_{3}\sum_{n=0}^{1} J^{n}\sum_{r:dist(s,i)=n}\   \nu\|\nabla_{s} f\|^2+\\  &   4 R_{3} \sum_{n=2}^{\infty}J^{n-2}\sum_{r:dist(r,i)=n}\   \nu\|\nabla_{r} f\|^2
\end{align}
 From (\ref{5eq22})-(\ref{5eq24})
we get the following bound for the terms for $n=2$ on the  sum  of (\ref{5eq19}) \begin{align}\nonumber \label{5eq25}R_{3}J^{2}\sum_{r:dist(s,i_{23})=2}\   \nu\|\nabla_{s} (\mathbb{E}^{i_3}f^{2})^{\frac{1}{2}}\|^2\leq  &   R_3J^{2}\sum_{s:dist(s,i)=2,4}\   \nu\|\nabla_{s} f\|^2+\\ \nonumber & 4R^{2}_{3}J^2 \sum_{n=2}^{\infty}J^{n}\sum_{r:dist(r,i)=n}\   \nu\|\nabla_{r} f\|^2+\\  &   4R^{2}_{3} \sum_{n=0}^{1}J^{n}\sum_{r:dist(r,i)=n}\   \nu\|\nabla_{r} f\|^2 \end{align}  
For every $s:dist(s,i_3)\geq 3$ we have $\nu\|\nabla_{s} (\mathbb{E}^{i_3}f^{2})^{\frac{1}{2}}\|^2\leq \nu\|\nabla_{s} f\|^2$, which gives 
\begin{align}\label{5eq26}\nonumber \sum_{n=3}^{\infty} J^{n}\sum_{r:dist(s,i_{23})=n}\   \nu\|\nabla_{s} (\mathbb{E}^{i_3}f^{2})^{\frac{1}{2}}\|^2\leq & \sum_{n=3}^{\infty} J^{n}\sum_{r:dist(s,i_{23})=n}\   \nu\|\nabla_{s} f\|^2\leq \\  & \sum_{n=3}^{\infty} J^{n-2}\sum_{r:dist(s,i)=n}\   \nu\|\nabla_{s} f\|^2
\end{align}since $dist(i_{23},i)=2$.  

Putting (\ref{5eq20}), (\ref{5eq21}), (\ref{5eq25}) and (\ref{5eq26}) in (\ref{5eq19}) leads to
 \begin{align}\nonumber \label{5eq27}\nu\|\nabla_{i_{23}}( \mathbb{E}^{\{ i_{2},i_3,i_4\}}f^{2})^{\frac{1}{2}}\|^{2} \leq & \nu\|\nabla_{i_{23}}(\mathbb{E}^{\{ i_{2},i_3\}}f^{2})^{\frac{1}{2}}\|^{2} \leq \\ & \nonumber      JR_4\nu\|\nabla_{i} f\|^2+      JR_4     \sum_{r:dist(s,i)=1}\   \nu\|\nabla_{s} f\|^2+ \\  &   R_{4} \sum_{n=2}^{\infty} J^{n-2}\sum_{r:dist(s,i)=n}\   \nu\|\nabla_{s} f\|^2\end{align}  
for $R_4=2R_{3}+5R^{2}_{3} $. The exact  same bound can be obtain for the other term on the  first  sum of the right hand side of (\ref{5eq17}).
Gathering together, (\ref{5eq17}), (\ref{5eq18}) and (\ref{5eq19})
leads to\begin{align}\nonumber\label{5eq28}
 \sum_{dist(r,i)=2}\nu\|\nabla_{r}( \mathbb{E}^{\{ i_{2},i_3,i_4\}}f^{2})^{\frac{1}{2}}\|^{2} \leq  &  JR_5\nu\|\nabla_{i} f\|^2+      JR_5     \sum_{r:dist(s,i)=1}\   \nu\|\nabla_{s} f\|^2+ \\  &   R_{5} \sum_{n=2}^{\infty} J^{n-2}\sum_{r:dist(s,i)=n}\   \nu\|\nabla_{s} f\|^2
  \end{align}
for $R_5=8R_{3}+R_4 $. 

Furthermore, for every $r:dist(r,i)\geq 3$,
\begin{align}\label{5eq29}\nu\|\nabla_{r} (\mathbb{E}^{\{ i_{2},i_3,i_4\}}f^{2})^{\frac{1}{2}}\|^2\leq\nu\|\nabla_{r} f\|^2\end{align} Finally, if we  put (\ref{5eq16}) and (\ref{5eq28}) and (\ref{5eq29}) in (\ref{5eq15}) we obtain 
 \begin{align}\label{5eq30}\nonumber\nu\| \nabla_{i}(\mathbb{E}^{\{\sim i\}}f^{2})^{\frac{1}{2}}\|^2=&\nu\|\nabla_{i}( \mathbb{E}^{\{i_{1} ,i_{2},i_3,i_4\}}f^{2})^{\frac{1}{2}}\|^{2}\leq R_{3}\nu\| \nabla_i (\mathbb{E}^{\{ i_{2},i_3,i_4\}}f^{2})^\frac{1}{2}
\|^2+ \\ &  +   R_{6} \sum_{n=0}^{\infty} J^{n}\sum_{r:dist(s,i)=n}\   \nu\|\nabla_{s} f\|^2        \end{align}
for constant $R_{6}=R_{3}+R_3R_5$.

If we repeat the same  calculation  recursively, for the first term on the right hand side of (\ref{5eq30}), then for   $\nu\|\nabla_{i}( \mathbb{E}^{\{i_3,i_4\}}f^{2})^{\frac{1}{2}}\|^{2}$ and $\nu\|\nabla_{i}( \mathbb{E}^{\{ i_4\}}f^{2})^{\frac{1}{2}}\|^{2}$ we will finally obtain  
 \begin{align*}\nonumber\nu\| \nabla_{i}(\mathbb{E}^{\{\sim i\}}f^{2})^{\frac{1}{2}}\|^2\leq  R_6 \sum_{n=3}^{\infty} J^{n}\sum_{r:dist(s,i)=n}\   \nu\|\nabla_{s} f\|^2  \end{align*}
for a constant $R_6$. From the last inequality and (\ref{5eq14})
we get 
 \begin{align*}\nu\| \nabla_{\Gamma_1}(\mathbb{E}^{\Gamma_{0}}f^{2})
^\frac{1}{2}\|^2\leq    C_1\nu\| \nabla_{\Gamma_i}f\|^2+C_2\nu\| \nabla_{\Gamma_j}f\|^2 \end{align*}
for a constant $C_{2}\leq JR_{6}(\sum_{k=0}^{\infty} (4J)^k)\leq\frac{JR_6}{1-4J} $ for $J$ sufficiently small such that $J<\frac{1}{4}$ and $J<\frac{1}{R_{6}-4}$ the proof of the proposition follows for $C_2<1$.

\end{proof}
 
  \section{log-Sobolev type inequalities.}\label{proof sec6}
   Since the purpose of this paper is to prove the log-Sobolev inequality for the infinite dimensional Gibbs measure without assuming the log-Sobolev inequality for the one site measure   $\mathbb{E}^{i,\omega}$, but the  weaker inequality for the   measure $\mu(dx_{i})=\frac{e^{\phi(x_{i})}dx_i}{\int \phi(x_{i})dx_{i}}$, we will show in this section that when the interactions are quadratic we can obtain a weaker log-Sobolev type inequality for the    $\mathbb{E}^{i,\omega}$ measure. This will be the object of the next proposition.

\begin{prop}\label{6prop1}  Assume   that the  measure $\mu $ satisfies the log-Sobolev inequality and  that the  local specification has quadratic  interactions $V$ as in (\ref{quadratic}). Then, for $J$ sufficiently small,   the one site measure  $\mathbb{E}^{i,\omega}$ satisfies the following log-Sobolev type inequality  \begin{align*} \nu\mathbb{E}^{i,\omega}\left(  f^2\log\frac{f^2}{\mathbb{E}^{i,\omega}f^2}\right)\leq c_{1}\nu\| \nabla_{i} f\|^2+c_{1} \sum_{n=1}^{\infty} J^{n-1}\sum_{r:dist(r,i)=n}\   \nu(\|\nabla_{r} f \|^2   )  \end{align*}
for some positive constant $c_{1}$.
\end{prop}          

\begin{proof} 
 Assume $g\geq 0$. We start with our main assumption that the measure $\mu(dx)=\frac{e^{\phi(x)}dx}{\int \phi(x)dx}$  satisfies the  log-Sobolev inequality for a constant $c_0$, that is    
\begin{align} \label{6eq1}\mu  (g^2\log\frac{g^2}{\mu g^2})\leq c_0\mu \| \nabla_i g \|^2\end{align}
We will interpolate this inequality to create the entropy with respect to the one site measure $\mathbb{E}^{ i,\omega}$ in the left hand side. For this we will  first define the  function
$$V^i=\sum_{ j\sim i}J_{i,j}V(x_{i},\omega_{j})$$
Notice that  $V^i\geq0$. Then inequality (\ref{6eq1}) for  $g=e^\frac{-V ^i}{2} f,f\geq0$ gives
 \begin{align}\nonumber\int e^{-\phi(x_i)}&(e^{-V^i} f^2\log\frac{e^{-V^i} f^2}{\int e^{-\phi(x_i)}(e^{-V^i} f^2)dx_{i} /  \int e^{-\phi(x_i)}
 dx_{i}})dx_{i}  \\  &
\ \ \ \ \ \ \  \   \ \ \ \  \   \   \    \    \    \ 
 \    \    \     \   \    \ \    \    \    \    \    \    \   \    \   \
  \label{6eq2}\leq c_{0}  \int e^{- \phi(x_i)} \| \nabla_{i} (e^\frac{-V^i}{2} f)
\|^2 dx_{i}\end{align}
 \noindent Denote  by $S_r$ and $S_l$ the right and left hand side of~\eqref{6eq2} respectively.
 If we use the Leibnitz rule
for the gradient on the right hand side of~\eqref{6eq2} we have

 \begin{align}\nonumber      S_r\leq&2 c_0\int  e^{-\phi(x_i)} \| e^\frac{{-V^i}}{2} ( \nabla_{i}  f)
\|^2  dx_{i}   \label{6eq3}+2 c_{0}  \int e^{-\phi(x_i)} \| f (\nabla_{i} e^\frac{{-V^i}}{2}  )
\|^2dx_{i}= \nonumber\\ &  \left( \int e^{-\phi(x_i)-V^i}f^2dx_{i} \right) c_02\left( \mathbb{E}^{ i,\omega}\|\nabla_{i} f
\|^2+\frac{ 1}{4}\mathbb{E}^{ i,\omega}f^{2}  \|  \nabla_{i}{V^i}
\|^2\right)\end{align}
On the left hand side of~\eqref{6eq2} we form  the entropy for the measure $\mathbb{E}^{ i,\omega}$ measure with Hamiltonian $\phi(x_i)+V^i$.
\begin{align*}\nonumber S_l = &\int e^{-\phi(x_i)-V^i} f^2log\frac{f^2}{\int e^{-\phi(x_i)-V^i}f^2dx_{i} / \int e^{-\phi(x_i)-V^i}
 dx_{i}} dx_{i}\nonumber \\ &+\int e^{-\phi(x_i)-V^i} f^2\log\frac{\left(\int e^{-\phi(x_i)}
 dx_{i}\right)  e^{-V^i}}{ \int e^{-\phi(x_i)-V^i}
 dx_{i}} dx_{i}\nonumber\\= &\nonumber( \int e^{-\phi(x_i)-V^i}f^2
 dx_{i} ) \left(\mathbb{E}^{ i,\omega}(f^2log\frac{f^2}{ \mathbb{E}^{ i,\omega}f^2}) - \mathbb{E}^{i,\omega}( f^2V^i)\right)\nonumber \\ &+\int e^{-\phi(x_i)-V^i} f^2log\frac{\int e^{-\phi(x_i)}
 dx_{i}   }{ \int e^{-\phi(x_i)-V^i}
 dx_{i}} dx_{i} \end{align*}
Since  $V^i$ is no negative,  the  last  equality leads to 
 \begin{align}\label{6eq4} S_l \geq( \int e^{-\phi(x_i)-V^i}f^2
 dx_{i} )\left(\mathbb{E}^{ i,\omega}(f^2log\frac{f^2}{ \mathbb{E}^{ i,\omega}f^2}) - \mathbb{E}^{ i,\omega}( f^2V^i)\right) \end{align}
 If we combine~\eqref{6eq2} together    with~\eqref{6eq3}  and~\eqref{6eq4} we obtain
\begin{align}\nonumber\mathbb{E}^{ i,\omega}(f^2log\frac{f^2}{\mathbb{E}^{ i,\omega}f^2}) \leq &2c_{0} \mathbb{E}^{ i,\omega}\| \nabla_{i} f
\|^2+\mathbb{E}^{ i,\omega}(f^2(\frac{ c_{0} \|  \nabla_{i}V^i
\|^2}{2}+V^i)) \nonumber \end{align}
If take the expectation with respect to the Gibbs  measure in  the last relationship we have
\begin{equation}\label{6eq5}\nu(f^2log\frac{f^2}{\mathbb{E}^{   i,\omega}f^2}) \leq 2 c_{0} \nu\| \nabla_{i} f
\|^2+\nu(f^2(\frac{ c_{0} \|  \nabla_iV^i
\|^2}{2}+V^i))\end{equation}
 From  \cite{B-Z} and \cite{R}   the  following estimate  
 of the entropy holds
 \begin{align}\label{6eq6}\mathbb{E}^{ i,\omega}(f^{2}log\frac{f^{2}}{\mathbb{E}^{ i,\omega}f^{2}}) \leq &  A\mathbb{E}^{ i,\omega}( f-\mathbb{E}^{ i,\omega}f )^{2} +
 \mathbb{E}^{ i,\omega}( f-\mathbb{E}^{ i,\omega}f )^{2}log\frac{( f-\mathbb{E}^{ i,\omega}f )^{2}}{\mathbb{E}^{ i,\omega}( f-\mathbb{E}^{ i,\omega}f )^{2}}\end{align}
  for some positive constant $A$.  If we take expectations with respect to  the Gibbs measure at the last  inequality we get
 \begin{align}\nu( \vert f \vert^2log\frac{ \vert f \vert^2}{\mathbb{E}^{ i,\omega} \vert f \vert^2}) \leq &A\nu\vert f-\mathbb{E}^{ i,\omega}f\vert^{2}\label{6eq7}+\nu(\vert f-\mathbb{E}^{ i,\omega}f\vert^2log\frac{\vert f-\mathbb{E}^{ i,\omega}f\vert^2}{\mathbb{E}^{ i,\omega}\vert f-\mathbb{E}^{ i,\omega}f\vert^2})\end{align}
 We can now use~\eqref{6eq5} to bound the second   term on the right hand side of~\eqref{6eq7}. Then we will obtain
 \begin{align*} \nonumber\nu(f^2&log \frac{f^2}{\mathbb{E}^{ i,\omega}f^2}) \leq A\nu \vert f -\mathbb{E}^{ i,\omega}f\vert^2  +2c_{0} \nu\| \nabla_{i} f
\|^2+\nu(\vert f-\mathbb{E}^{i,\omega}f\vert^{2} (\frac{ c_{0} \|  \nabla_i{V^i}
\|^2}{2}+V^i)) \leq  \\ &A\nu\vert f -\mathbb{E}^{ i,\omega}f\vert^2  +2c_{0} \nu\| \nabla_{i} f
\|^2+\sum_{ j\sim i}J_{i,j}\nu(\vert f-\mathbb{E}^{i,\omega}f\vert^{2} (2c_{0} \|  \nabla_i{V(x_{i},\omega_{j})}
\|^2+V(x_{i},\omega_{j})))\end{align*}
If we take under account  that we are considering quadratic interactions and bound $V$ and $\|  \nabla_i{V^i}
\|^2$ by (\ref{quadratic}) we get
 \begin{align*} \nonumber\nu(f^2&log \frac{f^2}{\mathbb{E}^{ i,\omega}f^2}) \leq  (A+4)\nu\vert f -\mathbb{E}^{ i,\omega}f\vert^2  +2c_{0} \nu\| \nabla_{i} f
\|^2+\\  &4(2c_{0}+1)kJ\nu(\vert f-\mathbb{E}^{i,\omega}f\vert^{2}  d^2(x_i) ) +(2c_{0}+1)kJ\sum_{ j\sim i} \nu(\vert f-\mathbb{E}^{i,\omega}f\vert^{2}  d^2(\omega_j) )\end{align*}
where above we also used  that $V(x_{i},\omega_{j})\leq 1+\left\vert V(x_{i},\omega_{j})\right\vert^2$. We can bound the first term on the right hand side by  Lemma \ref{teleutSpect1} 
and the third and the  fourth term   by Corollary \ref{4cor6}. 
 \begin{align*} \nonumber\nu(f^2log \frac{f^2}{\mathbb{E}^{ i,\omega}f^2}) \leq c_{1}\nu\| \nabla_{i} f\|^2+c_{1} \sum_{n=1}^{\infty} J^{n-1}\sum_{r:dist(r,i)=n}\   \nu(\|\nabla_{r} f \|^2   )   \end{align*}
which finishes the proof of  the proposition for $c_1=  (A+4)4(D_{L}+D_{3})+6(2c_{0}+1)k  D_{5}$.\end{proof}
    We  now prove a log-Sobolev type inequality for the product measure $\mathbb{E}^{ \Gamma_{k}}$ for $k=0,1$.\begin{prop} \label{6prop2}
Assume    that  the   measure $\mu $ satisfies the log-Sobolev inequality and  that the  local specification has quadratic interactions $V$ as in (\ref{quadratic}). Then, for $J$ sufficiently small,  the following log-Sobolev type inequality for the product measures  $\mathbb{E}^{\Gamma_k,\omega}$ holds  
 \begin{align*} \nu\mathbb{E}^{ \Gamma_{k}}(f^2log\frac{f^2}{\mathbb{E}^{ \Gamma_{k}} f^2}) \leq\tilde C \nu\| \nabla_{\Gamma_0}
f
\|^2+\tilde C\nu\| \nabla_{\Gamma_1} f
\|^2\end{align*}for $k=0,1$, and  some positive constant $\tilde C$.
\end{prop}
\begin{proof} We will prove Proposition~\ref{6prop2} for $k=1$,
 that is  $$\nu\mathbb{E}^{ \Gamma_{1}}(f^2log\frac{f^2}{\mathbb{E}^{ \Gamma_{1}}f^2}) \leq\tilde C \nu\| \nabla_{\Gamma_0}
f
\|^2+\tilde C\nu\| \nabla_{\Gamma_1} f
\|^2$$ for $f\geq 0$.  

In the proof of this proposition we will use the following estimation. For any $i\in \mathbb{Z}$ and $\{ \sim i \}=i_1,i_2,i_3,i_4$ denote
\begin{align*}\Theta(i):&=\nu\|\nabla_{i}( \mathbb{E}^{\{i_{1} ,i_{2},i_3,i_4\}}f^{2})^{\frac{1}{2}}\|^{2} +\nu \|\nabla_{i}( \mathbb{E}^{\{ i_{2},i_3,i_4\}}f^{2})^{\frac{1}{2}}\|^{2}+\\ &+\nu\|\nabla_{i}( \mathbb{E}^{\{ i_3,i_4\}}f^{2})^{\frac{1}{2}}\|^{2}+\nu\|\nabla_{i}( \mathbb{E}^{\{ i_4\}}f^{2})^{\frac{1}{2}}\|^{2}\end{align*}
and
\begin{align*}\Lambda(i):=  \sum_{n=0}^{\infty} J^{n}\sum_{r:dist(s,i)=n}\   \nu\|\nabla_{s}  f\|^2 \end{align*} From the calculations of the components of  the sum of $\Theta(i)$ in the proof of   Proposition \ref{5lem3}  and the recursive inequality (\ref{5eq30}) we can surmise that there exists an $R_7>0$   such that   \begin{align}\label{6eq8} \Theta(i) \leq  R_{7}\Lambda(i)
\end{align}  We start with the  following enumeration of the nodes in $\Gamma_1$ as depicted in figure \ref{fig2}.
\begin{figure}[h]
           \begin{center}
\epsfig{file=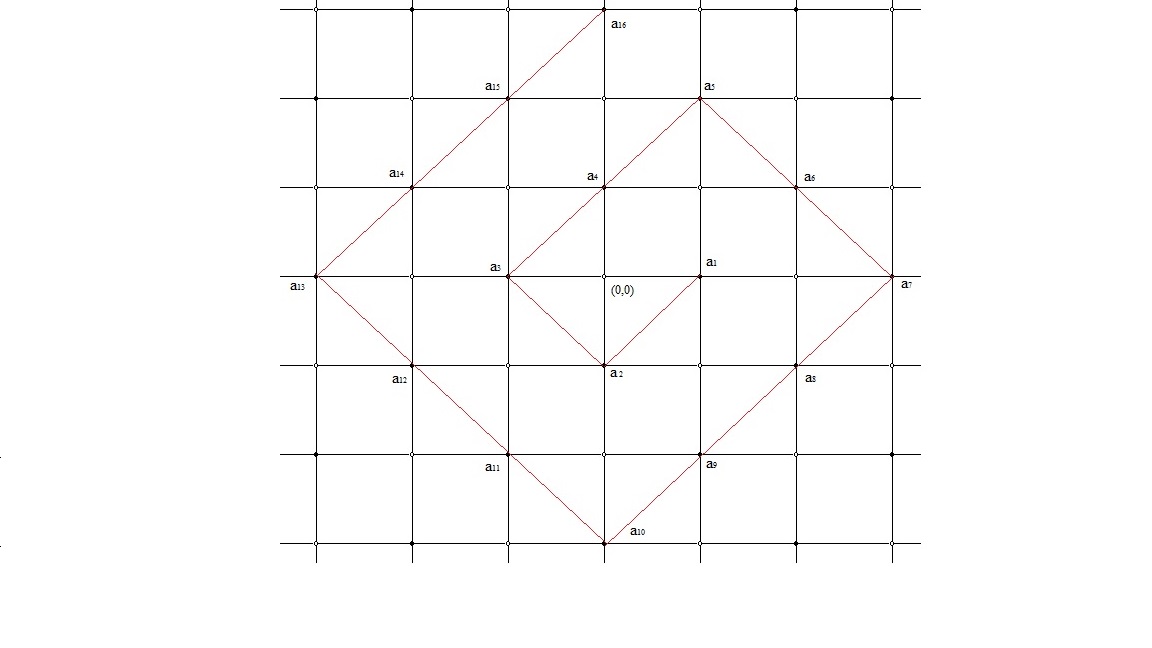, height=4.5cm}
\caption{$\circ = \Gamma_0$, $\bullet = \Gamma_1$}
\label{fig2}
           \end{center}
\end{figure}
 Denote the nodes in $\Gamma_1$ closest to $(0,0)$, that is the neighbours of $(0,0)$ and name them $a_1,a_2,a_3,a_4$,  with $a_1$ being any of the four and the rest named clockwise. Then choose any of the nodes in $\Gamma_1$ of distance two from $a_4$ and distance three from $(0,0)$, and name it $a_5$ and continue clockwise the enumeration with the rest of the nodes in $\Gamma_1$ of distance three form $(0,0)$. Then the same for the nodes of $\Gamma_1$ of distance four from $(0,0)$. We continue with the same way with the nodes in $\Gamma_1$ of higher distances from $(0,0)$, moving clockwise while we move away  from $(0,0)$. In this way the nodes in $\Gamma_1$ are enumerated in a spiral way moving clockwise away from $(0,0)$.   In that way we can write $\mathbb{E}^{ \Gamma_{1}}=\sqcap_{i=1}^{+\infty}\ \mathbb{E}^{ a_{i}}$.

Then
  the entropy of the product measure $\mathbb{E}^{ \Gamma_{1}}$ can be calculated by being expressed in terms of the entropies of single nodes in $\Gamma_1$ for which we have shown a log-Sobolev type inequality
in Proposition \ref{6prop1}.
\begin{align}\nu\mathbb{E}^{ \Gamma_1}(f^2\log\frac{f^2}{\mathbb{E}^{ \Gamma_1}f^2}) = \label{6eq9}\sum_{k=1}^{+\infty}\nu\mathbb{E}^{a_k}(\mathbb{E}^{ a_{k-1}}...\mathbb{E}^{a_1}f^2\log\frac{\mathbb{E}^{a_{k-1}}...\mathbb{E}^{a_1}f^2}{\mathbb{E}^{a_k}...\mathbb{E}^{a_1}f^2})\end{align}
To compute the entropies in the right hand side of (\ref{6eq9})  we will use the log-Sobolev type inequality for the one  site measure $\mathbb{E}^{k }$ from Proposition  \ref{6prop1}.
 \begin{align}\label{6eq10} \nu\mathbb{E}^{  a_{k}}(\mathbb{E}^{ a_{k-1}}...\mathbb{E}^{ a_{1}}f^2\log\frac{\mathbb{E}^{a_{k-1}}...\mathbb{E}^{a_{1}}f^2}{\mathbb{E}^{a_{k}}...\mathbb{E}^{a_{1}}f^2}) \leq & c_{1}\nu\| \nabla_{a_{k}} f\|^2+ \\ \nonumber & c_{1} \sum_{n=1}^{\infty} J^{n-1}\sum_{j:dist(j,a_{k})=n}\   \nu\|\nabla_{j} (\mathbb{E}^{ a_{k-1}}...\mathbb{E}^{ a_{1}}f^2)^\frac{1}{2}\|^2 \end{align}
 where above we used that $\nu\|\nabla_{a_k} (\mathbb{E}^{ a_{k-1}}...\mathbb{E}^{ a_{1}}f^2)^\frac{1}{2}\|^2 \leq \nu\|\nabla_{a_k} f\|^2$, since  by the way the spiral was constructed it's elements do not neighbour  with each other. For every $j$ that neighbours with at least one of the $a_{k-1},a_{k-2},...,a_1$ we have  that 
\begin{align*}\nu\|\nabla_{j} (\mathbb{E}^{ a_{k-1}}...\mathbb{E}^{ a_{1}}f^2)^\frac{1}{2}\|^2 \leq\Theta(j)
\end{align*}
which because of (\ref{6eq8}) implies
\begin{align}\label{6eq11}\nu\|\nabla_{j} (\mathbb{E}^{ a_{k-1}}...\mathbb{E}^{ a_{1}}f^2)^\frac{1}{2}\|^2 \leq R_{7}\Lambda(j)
\end{align}
For every $j$ that does not neighbour with any of the $a_{k-1},a_{k-2},...,a_1$ we have  that 
\begin{align}\label{6eq12}\nu\|\nabla_{j} (\mathbb{E}^{ a_{k-1}}...\mathbb{E}^{ a_{1}}f^2)^\frac{1}{2}\|^2 \leq \nu\|\nabla_{j} f\|^2 \end{align}
Putting   (\ref{6eq11}) and (\ref{6eq12}) in (\ref{6eq10}) we obtain   \begin{align*} \nonumber\nu\mathbb{E}^{  a_{k}}(\mathbb{E}^{ a_{k-1}}...\mathbb{E}^{ a_{1}}f^2\log\frac{\mathbb{E}^{a_{k-1}}...\mathbb{E}^{a_{1}}f^2 }{\mathbb{E}^{a_{k}}...\mathbb{E}^{a_{1}}f^2}) \leq  & c_{1}\nu\| \nabla_{a_{k}} f\|^2+\\ & c_{1} R_7\sum_{n=1}^{\infty} J^{n-1}\sum_{j:dist(j,a_{k})=n}\Lambda(j)+\\ &  c_{1} \sum_{n=1}^{\infty} J^{n-1}\sum_{j:dist(j,a_{k})=n}   \nu\|\nabla_{j} f\|^2  \end{align*}
 Combining this with (\ref{6eq9})  \begin{align*}\nu\mathbb{E}^{ \Gamma_1}(f^2\log\frac{f^2}{\mathbb{E}^{ \Gamma_1}f^2}) \leq  & c_{1}\sum_{k=1}^{+\infty}\nu\| \nabla_{a_{k}} f\|^2+\\ & c_{1} R_7\sum_{k=1}^{+\infty}\sum_{n=1}^{\infty} J^{n-1}\sum_{j:dist(j,a_{k})=n}\Lambda(j)+\\ &  c_{1}R_7\sum_{k=1}^{+\infty} \sum_{n=1}^{\infty} J^{n-1}\sum_{j:dist(j,a_{k})=n}   \nu\|\nabla_{j} f\|^2\end{align*}
Since in the last two sums above, for every $j \in \mathbb{Z}^2$, the terms $\Lambda(j)$ and  $\nu\|\nabla_{j} f\|^2$ appear one time for every $i\in \Gamma_1$ with a coefficient $J^{dist({i,j)-1}}$, the accompanying coefficient for any of these terms  is $\frac{\sum_{k=0}^{\infty}(4J)^k}{J}$ since for every  node $j$, $\#\{i: dist(i,j)=n\}\leq 4^n$. So by rearranging the terms in the last inequality we get \begin{align*}\nu\mathbb{E}^{ \Gamma_1}(f^2\log\frac{f^2}{\mathbb{E}^{ \Gamma_1}f^2}) \leq  & c_{1}\nu\| \nabla_{\Gamma_1} f\|^2+\frac{    R_7c_{1}}{J} \sum_{j\in \mathbb{Z}^2}\left( \sum_{n=0}^{\infty}(4J)^{n}\right) \Lambda(j)+\\ & \frac{R_7c_{1}}{J} \sum_{j\in \mathbb{Z}^2} \sum_{n=0}^{\infty}(4J)^{n} \nu\|\nabla_{j} f\|^2\\ =&  c_{1}\nu\| \nabla_{\Gamma_1} f\|^2+ \frac{R_7c_{1}}{J(1-4J)} \sum_{j\in \mathbb{Z}^2}  \Lambda(j)+\frac{R_7c_{1}}{J(1-4J)} \sum_{j\in \mathbb{Z}^2}  \nu\|\nabla_{j} f\|^2\end{align*}
for  $J<\frac{1}{4}$.   For the first sum 
\begin{align*}\sum_{j\in \mathbb{Z}^2}\Lambda(j)&=  \sum_{j\in \mathbb{Z}^2}\sum_{n=0}^{\infty} J^{n}\sum_{r:dist(s,j)=n}\   \nu\|\nabla_{s}  f\|^2\leq \sum_{j\in \mathbb{Z}^2} \left(\sum_{n=0}^{\infty}(4J)^{n}\right) \nu\|\nabla_{j} f\|^2\\  &=\frac{1}{1-4J}\sum_{j\in \mathbb{Z}^2}   \nu\|\nabla_{j} f\|^2
\end{align*} We finally get  \begin{align*}\nu\mathbb{E}^{ \Gamma_1}(f^2\log\frac{f^2}{\mathbb{E}^{ \Gamma_1}f^2}) \leq  c_{1}\nu\| \nabla_{\Gamma_1} f\|^2+ \frac{2R_7c_{1}}{J(1-4J)^{2}} \sum_{j\in \mathbb{Z}^2}  \nu\|\nabla_{j} f\|^2\end{align*}
 \end{proof}

\section{spectral gap and proof of main theorem}\label{spectralgap}
 In    Proposition~\ref{6prop1} and Proposition~\ref{6prop2}  we showed 
a log-Sobolev type inequality for the one
 site measure $\mathbb{E}^{i,\omega}$and then obtained a similar inequality through it for the product measures $\mathbb{E}^{ \Gamma_{i},\omega},i=0,1$.
 In   Lemma~\ref{teleutSpect1} a spectral gap  type inequality was also shown for the one site measure  $\mathbb{E}^{i,\omega}$ for  both cases.  In the following proposition  the spectral gap type inequality of Lemma~\ref{teleutSpect1}  will   
be extended
to the product measure $\mathbb{E}^{\Gamma_{i},\omega},i=0,1$. However, this does not happen through the Spectral Gap type inequality for the one site measure   $\mathbb{E}^{i,\omega}$ of Lemma \ref{teleutSpect1}  but through the log-Sobolev type inequality for $\mathbb{E}^{ \Gamma_{i},\omega},i=0,1$ of Proposition \ref{6prop2}. 
\begin{prop}\label{7prop1}Assume   that the  measure $\mu $ satisfies the log-Sobolev inequality and  that the  local specification has quadratic interactions $V$ as in (\ref{quadratic}). Then, for $J$ sufficiently small,  the following
spectral gap type inequality holds  
 $$\nu\vert f-\mathbb{E}^{ \Gamma_{i},\omega}f\vert^2\ \leq \tilde R\nu\| \nabla_{\Gamma_{0}}
f\|^2+ \tilde R\nu\| \nabla_{\Gamma_{1}}
f\|^2$$for $i=0,1$ and  some positive constant $\tilde R$.\end{prop}
The proof of this proposition is based in the use of  the  log-Sobolev type inequality for the product measures $\E^{\Gamma_k},k=0,1$ of Proposition \ref{6prop2} and the sweeping out relations for the same measures from Proposition \ref{4prop8}. The proof of this proposition is presented in Lemma 7.1 of \cite{Pa1}. 

Spectral Gap inequalities have been associated with convergence to equilibrium and ergodic properties. In the next proposition  we will use the weaker spectral gap inequality for the product measures  $\mathbb{E}^{ \Gamma_i,\omega},i=0,1$ of 
 Proposition \ref{7prop1} to show  the a.e. convergence of $ \mathcal{P}^n  $ to the infinite dimensional Gibbs measure $\nu$, where $\mathcal{P}^n$ is defined as follows
  \begin{align}\label{7eq1} \mathcal{P}^nf=\begin{cases}  
   f & n =0 \\
 \mathbb{E}^{\Gamma_{0}}\mathcal{P}^{n-1}f & n \text{\ odd} \\
\mathbb{E}^{\Gamma_{1}}\mathcal{P}^{n-1}f & n \text{\ even\ }>0
\end{cases} \end{align}
 \begin{prop}\label{7prop2}Assume   that the  measure $\mu $ satisfies the log-Sobolev inequality and  that the  local specification has quadratic interactions $V$ as in (\ref{quadratic}). Then, for $J$ sufficiently small, and  $\mathcal{P}$ as in (\ref{7eq1}),    $\mathcal{P}^n f$ converges $\nu$-a.e. to the Gibbs measure $\nu$.
 \end{prop}
 \begin{proof} We will follow closely  \cite{G-Z}. 
 We will  compute the variance of the   $\mathcal{P}^nf$ with respect to the product measure $\mathbb{E}^{\Gamma_k}$ for $k=0$ or $1$ when $n$ is odd or even respectively. For this we will use the spectral gap type inequality for the product measures    $\mathbb{E}^{ \Gamma_i,\omega},i=0,1$
presented in Proposition \ref{7prop1}. 
\begin{align*}\nonumber\nu\vert\mathcal{P}^n f- \mathcal{P}^{n+1} f\vert^2=&\nu\vert\mathcal{P}^n f- \mathbb{E}^{\Gamma_k}\mathcal{P}^{n} f\vert^2\leq   \tilde R\nu\| \nabla_{\Gamma_0} \mathcal{P}^n f
\|^2+  \tilde R \nu\| \nabla_{\Gamma_1} \mathcal{P}^nf
\|^2= \\ = & \tilde R \nu\| \nabla_{\Gamma_k} \mathcal{P}^nf\|^2=\tilde R \nu\| \nabla_{\Gamma_k} \mathbb{E}^{\Gamma_{1-k}} \mathcal{P}^{n-1}f\|^2\end{align*} 
 where $k$ above is $0$ or $1$ if $n$ is odd or even respectively. If we use the first sweeping out inequality of Proposition~\ref{4prop8} we get
\begin{align*}\nonumber\nu\vert\mathcal{P}^n f- \mathcal{P}^{n+1} f\vert^2\leq&\tilde R R_1\nu\| \nabla_{\Gamma_k}  \mathcal{P}^{n-1}f\|^2+\tilde R R_2\nu\| \nabla_{\Gamma_{1-k}} \mathcal{P}^{n-1}f\|^2=\tilde R R_2\nu\| \nabla_{\Gamma_{1-k}} \mathcal{P}^{n-1}f\|^2\end{align*} 
where we recall $R_2\leq \frac{JG_6}{1-4J} <1$. If we apply     $n-2$ more times   Proposition~\ref{4prop8} we obtain the following bound
\begin{align}\label{7eq2} \nu\vert\mathcal{P}^n f- \mathcal{P}^{n+1} f\vert^2\leq\tilde R R_2^{n-2}\left(R_1\nu\| \nabla_{\Gamma_1}f\|^2 +R_2\nu\| \nabla_{\Gamma_0} f\|^2\right)\end{align}
which converges to zero as $n$  goes   to infinity, because
 $R_2<1$. If we define the sets $$\Delta_n=\{\vert \mathcal{P}^nf- \mathcal{P}^{n+1} f\vert \geq \frac{1}{2^{n}}\}$$ we can calculate 
\begin{equation*}\nu (\Delta_n)=\nu\left( \left\{\vert \mathcal{P}^nf- \mathcal{P}^{n+1} f\vert \geq \frac{1}{2^{n}}\right\} \right)\leq 2^{2n}\nu \vert\mathcal{P}^nf- \mathcal{P}^{n+1} f\vert^2  \end{equation*}by Chebyshev inequality. If we use (\ref{7eq2}) to bound the last one we get   
\begin{equation*}\nu  (\Delta_n)\leq (4R_2)^{n-2}8 \tilde R \left(R_1\nu\| \nabla_{\Gamma_1}f\|^2 +R_2\nu\| \nabla_{\Gamma_0} f\|^2\right)  \end{equation*}and for $J$ sufficiently small  such that $4R_2\leq\frac{4JG_6}{1-4J}<\frac{1}{2}$ (recall that $R_2\leq\frac{JG_6}{1-4J}$) we have that \begin{equation*}\sum_{n=0}^\infty\nu (\Delta_n)\leq\left(\sum_{n=0}^\infty\frac{1}{2^n} \right) 64\tilde R  \left(R_1\nu\| \nabla_{\Gamma_1}f\|^2 +R_2\nu\| \nabla_{\Gamma_0} f\|^2\right)  <\infty
\end{equation*}   Thus, $$\{ \mathcal{P}^n f-\nu \mathcal{P}^n f\}_{n \in \mathbb{N}}$$ 
 converges $\nu-$almost surely  by the  Borel-Cantelli lemma.  Furthermore,  $$\mathcal{P}^nf\rightarrow \vartheta\ (f) \   \   \   \   \   \nu-\text{a.e.}$$ We will first show that $\theta (f)$  is a constant, which means that it does not depend on variables on $\Gamma_0$ or $\Gamma_1$. We first notice that  $\mathcal{P}^n(f)$ is a function on $\Gamma_0$ or $\Gamma_1$ when $n$  is odd or even respectively. This implies that the limits  $$\vartheta_{o}(f):=\lim_{n \text{\ odd}, n\rightarrow \infty}\mathcal{P }^nf \text{ \ and \ }\vartheta_e(f):=\lim_{n \text{\ even}, n\rightarrow \infty}\mathcal{P}^nf$$ do not depend on  variables on $\Gamma_0$ and $\Gamma_1$ respectively.    However, since  the two  subsequences $\{\mathcal{P}^nf\}_{n\text{\ even}}$ and $\{\mathcal{P}^nf\}_{n\text{\ odd}}$  converge to $\vartheta(f)$ $\nu-$a.e.  we conclude that $$\vartheta_{o}(f)=\vartheta(f)=\vartheta_{e}(f)$$
 which implies that $\theta (f)$ is a constant. From that we obtain that \begin{equation}\label{7eq3} \nu \left(\vartheta (f) \right)=\vartheta(f)
\end{equation} 
Since the sequence $\{ \mathcal{P}^n f\}_{n \in \mathbb{N}}$ 
 converges $\nu-$almost, the same holds for  the sequence $\{ \mathcal{P}^n f-\nu \mathcal{P}^n f\}_{n \in \mathbb{N}}$. 

It remains to show that $\vartheta(f)=\nu (f)$. At first we show this for  positive bounded functions  $f$. In this case we have 
 \begin{align}\label{7eq4}\lim_{n\rightarrow \infty}(\mathcal{P}^n f-\nu \mathcal{P}^n f)=\vartheta (f)-\nu\left(\vartheta(f) \right) =\vartheta (f)-\vartheta (f)=0\end{align}
 by the dominated     convergence theorem and  (\ref{7eq3}). On the other hand, we also have 
 \begin{align}\label{7eq5}\lim_{n\rightarrow \infty}( \mathcal{P}^n f-\nu \mathcal{P}^n f)=\lim_{n\rightarrow \infty}(\mathcal{P}^n f-\nu f)=\vartheta (f)-\nu (f)
\end{align} 
where above  we used the definition of the Gibbs measure $\nu$. From (\ref{7eq4}) and (\ref{7eq5}) we get that 
$$\vartheta (f)=\nu (f)$$
for bounded functions $f$. We now extend it to   no bounded positive functions $f$. Consider $f_k(x):=\max\{f(x),k\}$ for any $k\in \mathbb{N}$. Then  $$\vartheta(f_{k})=\lim_{n\rightarrow \infty}\mathcal{P}^n f_{k}= \nu f_{k},  \  \   \nu \  \  a.e. $$since $f_k(x)$ is bounded by $k$. 
 Then since $f_k$ is increasing on $k$, by the monotone convergence theorem we   obtain 
\begin{align*}\vartheta(f)=\lim_{k\rightarrow \infty}\vartheta(f_k)=\lim_{k\rightarrow \infty}\nu(f_{k})=\nu(\lim_{k\rightarrow \infty}f_{k})=\nu(f)  \  \   \nu \  \  a.e. \end{align*}
The assertions then can be extended to no positive functions $f$ by writing $f=f^+-f^-$, where $f^+=\max\{f,0\}$ and $f^-=-\min\{f,0\}$.  
 \end{proof} 
 We can now proceed with the proof of the main theorem.
\subsection{proof Proposition \ref{proposition}\label{finalproof}}
The  proof of the main result will be based on the iterative method developed by  
 Zegarlinski in \cite{Z1} and \cite{Z2} (see also \cite{Pa1} and \cite{I-P} for similar application). We will start with a lemma that shows the iterative step.

Denote  $Ent_\mu (f):=\mu(f^2\log\frac{f^2}{\mu f^2})$ the entropy of a function $f$ with respect to a measure $\mu$. 
\begin{lem}
Assume  $\mathcal{P}$ as in (\ref{7eq1}). For any   $n \geq 1$,
\begin{align}
\label{7eq6}
\mathcal{P}^n[f\log f ] 
= &\sum_{ m=0}^{n-1} \mathcal{P}^{n-m-1}  [Ent_{\E^{\Gamma_k}}(\mathcal{P}^m f)]
+\mathcal{P}^n f\log \mathcal{P}^n f
\end{align}where $k$ above is $0$ or $1$ if $n$ is odd or even respectively.
\end{lem}
\begin{proof}

One observes that for any $\Lambda\subset \Z^2 $
\begin{equation}
\label{7eq7}
Ent_{\E^\Lambda}(g) =\mathbb{E}^{\Lambda}\left(g\log \frac{g}{\mathbb{E}^{\Lambda}g}\right)=\E^\Lambda[g \log g]  -\\ (\E^\Lambda g)\log (\E^\Lambda g)
\end{equation}
 The statement (\ref{7eq6})  for $n=1$ can  be trivially derived from (\ref{7eq7}) if we put $\Lambda=\Gamma_0$ and $g= f$. Assuming (\ref{7eq6}) is
true for some $n \geq 0$, we prove it for $n+1$.
  Apply \eqref{7eq7} with $\Lambda = \Gamma_k$ and $\mathcal{P}^n f$ in the place of $g$, where $k$ above is $0$ or $1$ if $n$ is odd or even respectively:\begin{align*}
 \E^{\Gamma_k}[ (\mathcal{P}^n f)\log  (\mathcal{P}^n f)] &= Ent_{\E^{\Gamma_k}}(\mathcal{P}^n f) + (\E^{\Gamma_k} \mathcal{P}^n f)\log  (\E^{\Gamma_k} \mathcal{P}^n f)\\ &=Ent_{\E^{\Gamma_k}}(\mathcal{P}^n f) + (  \mathcal{P}^{n+1} f)\log  (  \mathcal{P}^{n+1} f)
\end{align*}
Using this, and applying $\E^{\Gamma_k}$ to \eqref{7eq6} we   obtain
\eqref{7eq6} for $n+1$.   \end{proof}

Using Proposition \ref{7prop2} we have $\mathcal{P}^n[f\log f] \to \nu[f\log f]$
and $\mathcal{(P}^n f)\log(\mathcal{(P}^n f) \to \nu[f]\log \nu[f]$, $\nu$-a.e. From this
and Fatou's lemma, \eqref{7eq6} gives
\begin{align}
Ent_\nu(f^{2}) &\leq \liminf_{n \to \infty}
\bigg\{
\nu\bigg[
\sum_{ m=0}^{n-1} \mathcal{P}^{n-m-1}  [Ent_{\E^{\Gamma_k}}(\mathcal{P}^m f)]
\bigg]
\bigg\}        \nonumber
\\
&= \liminf_{n \to \infty}\bigg\{
\sum_{m=0}^{n-1} \nu[Ent_{\E^{\Gamma_k}}(\mathcal{P}^m f^{2})]
\label{7eq8}
\bigg\}
\end{align}
where we used the fact that $\nu$ is a Gibbs measure to obtain
the last equality.
If we use Proposition \ref{6prop2} to bound the first term of the first sum we have
\[
\nu Ent_{\E^{\Gamma_0}}(f^2) \leq \tilde C \nu\|\nabla_{\Gamma_0}
f
\|^2+\tilde C\nu\| \nabla_{\Gamma_1} f
\|^2
\]
  Similarly, for $m \ge 1$, we can use Proposition \ref{6prop2}
and then we get
\begin{align*}
\nu[Ent_{\E^{\Gamma_k}}(\mathcal{P}^{m} f^2)]
&\le \tilde C\nu \|\nabla_{\Gamma_k} \sqrt{\mathcal{P}^m f^2}\|^2
\le \tilde C [C_1 C_2^{m-1}  \nu \|\nabla_{\Gamma_1} f\|^2 
+C_2^{m} \nu \|\nabla_{\Gamma_0} f\|^2 ]
\end{align*}
where, for the last inequalities we used Proposition  \ref{5lem3} and induction.
Substituting in \eqref{7eq8}, we obtain (recall that $0< C_2 < 1$)
\[
Ent_\nu(f^2) \leq 
\frac{\tilde C C_1}{C_{2}(1-C_2)} \nu \|\nabla_{\Gamma_1} f\|^2 
+ \frac{\tilde C}{1-C_2} \nu \|\nabla_{\Gamma_0} f\|^2 
\le \overline{C} \, \nu \|\nabla f\|^2
\]
where $\overline C$ is the largest of the two coefficients.
This ends the proof of  the log-Sobolev inequality for $\nu$.
   
   \qed

\bibliographystyle{alpha}

\end{document}